\documentclass[11pt,a4paper,reqno]{amsart}%
\usepackage{amssymb}
\usepackage{amsmath}
\usepackage{amsfonts}
\usepackage{color}
\usepackage{bbm}
\usepackage{graphicx}
\usepackage{rotating}
\usepackage{hyperref}
\usepackage{mathrsfs}
\usepackage[all]{xy}%
\providecommand{\U}[1]{\protect\rule{.1in}{.1in}}
\newtheorem{theorem}{Theorem}

\newtheorem{corollary}[theorem]{Corollary}

\newtheorem{definition}[theorem]{Definition}
\newtheorem{example}[theorem]{Example}

\newtheorem{lemma}[theorem]{Lemma}

\newtheorem{proposition}[theorem]{Proposition}
\newtheorem{remark}[theorem]{Remark}

\thanks{}
\email{lvitagliano@unisa.it}
\begin{document}
\title{On The Strong Homotopy Associative Algebra of a Foliation}
\author{\textsc{{Luca Vitagliano}}}
\address{DipMat, Universit\`a degli Studi di Salerno, \& Istituto Nazionale di Fisica
Nucleare, GC Salerno, Via Ponte don Melillo, 84084 Fisciano (SA), Italy.}

\begin{abstract}
An involutive distribution $C$ on a smooth manifold $M$ is a Lie-algebroid
acting on sections of the normal bundle $TM / C$. It is known that the
Chevalley-Eilenberg complex associated to this representation of $C$ possesses
the structure $\mathbb{X}$ of a strong homotopy Lie-Rinehart algebra. It is natural to
interpret $\mathbb{X}$ as the (derived) Lie-Rinehart algebra of vector fields on the
space $\boldsymbol{P}$ of integral manifolds of $C$. In this paper, I show that
$\mathbb{X}$ is embedded in an $A_{\infty}$-algebra $\mathbb{D}$ of (normal)
differential operators. It is natural to interpret $\mathbb{D}$ as the
(derived) associative algebra of differential operators on $\boldsymbol{P}$.
Finally, I speculate about the interpretation of $\mathbb{D}$ as the
\emph{universal enveloping strong homotopy algebra} of $\mathbb{X}$.
\end{abstract}

\maketitle

\section*{Introduction}

Let $M$ be a finite dimensional smooth manifold and $C$ an involutive
distribution on it. In view of Fr\"{o}benius theorem the datum of $C$ is
equivalent to the datum of a foliation of $M$. The pair $(M,C)$ is a finite
dimensional instance of a \emph{diffiety }(or a $D$\emph{-scheme}, in the
algebraic geometry language) which is a geometric object formalizing the
concept of \emph{partial differential equation}. There is a rich cohomological
calculus, sometimes called \emph{secondary calculus} \cite{v84,v98,v01},
associated to a diffiety $(M,C)$. Secondary calculus may be interpreted to
some extent as a differential calculus on the space of integral manifolds of
$C$. All constructions of standard calculus on manifolds (vector fields,
differential forms, differential operators, etc.) have a secondary analogue,
i.e., a formal analogue within secondary calculus. For instance,
\emph{secondary functions} are characteristic cohomologies of $C$,
\emph{secondary vector fields} are characteristic cohomologies with local
coefficients in normal vector fields, etc. (see the first part of \cite{v09}
for a compact review of \emph{secondary Cartan calculus}). In \cite{vi12} I
speculated that secondary calculus is actually a \emph{derived differential
calculus} in the sense that \textquotedblleft\emph{all secondary constructions
come from suitable algebraic structures up to homotopy at the level of
(characteristic) cochains}\textquotedblright. As a fundamental motivation
behind this conjecture, I discussed in \cite{vi12} the strong homotopy
Lie-Rinehart algebra of secondary vector fields.

This is a companion paper of \cite{vi12}. Here, I present a further motivation
behind the above mentioned conjecture: the $A_{\infty}$-algebra of secondary
(linear, scalar) differential operators. The main technical tools to show the
existence of such $A_{\infty}$-algebra are homological perturbations and
homotopy transfer. The strategy of the proof is the following. Let
$\mathcal{D}(\overline{\Lambda})$ be the associative differential graded (DG) algebra of
differential operators on longitudial differential forms $\overline{\Lambda}$ (i.e., differential
forms along $C$). It projects naturally onto the DG module $\overline{\Lambda
}\otimes\overline{\mathcal{D}}$ of $\overline{\Lambda}$-valued differential
operators on $C^{\infty}(M)$, normal to $C$. Actually, there are
\emph{contraction data} for $\mathcal{D}(\overline{\Lambda})$ over
$\overline{\Lambda}\otimes\overline{\mathcal{D}}$ (see Subsection
\ref{SecHPHT} for the definition of contraction data). The latter allow to
induce an $A_{\infty}$-algebra structure on $\overline{\Lambda}\otimes
\overline{\mathcal{D}}$ from the DG algebra structure on $\mathcal{D}%
(\overline{\Lambda})$. Suitable contraction data can be constructed using
purely geometric (supplementary) data as follows. First construct
\emph{Poincar\'{e}-Birkhoff-Witt} (\emph{PBW}) type isomorphisms
$\mathcal{D}(\overline{\Lambda})\approx S^{\bullet}\mathrm{Der}\overline
{\Lambda}$ and $\overline{\Lambda}\otimes\overline{\mathcal{D}}\approx
\overline{\Lambda}\otimes S^{\bullet}\overline{\mathfrak{X}}$ (here
$\mathrm{Der}\overline{\Lambda}$ is the DG Lie-Rinehart algebra of derivations
of $\overline{\Lambda}$, and $\overline{\mathfrak{X}}$ is the module of
sections of the normal bundle $TM/C$). Second, notice that $S^{\bullet
}\mathrm{Der}\overline{\Lambda}$ and $\overline{\Lambda}\otimes S^{\bullet
}\overline{\mathfrak{X}}$ are commutative DG algebras and there are simple
contraction data for $S^{\bullet}\mathrm{Der}\overline{\Lambda}$ over
$\overline{\Lambda}\otimes S^{\bullet}\overline{\mathfrak{X}}$. Third, use the
Homological Perturbation Theorem (and the PBW isomorphisms) to construct
contraction data for $\mathcal{D}(\overline{\Lambda})$ over $\overline
{\Lambda}\otimes\overline{\mathcal{D}}$, from contraction data for
$S^{\bullet}\mathrm{Der}\overline{\Lambda}$ over $\overline{\Lambda}\otimes
S^{\bullet}\overline{\mathfrak{X}}$.

The paper is basically self-consistent and it is organized as follows. It is
devided into three sections. In the first one, I collect the algebraic
preliminaries: namely, differential operators on graded algebras, strong
homotopy structures, homological perturbations and homotopy transfer. In
subsection \ref{SecHTUE}, I show how, under suitable regularity conditions
(namely, the existence of a PBW type isomorphism), the universal enveloping
algebra of a DG Lie-Rinehart algebra contracting over a complex $(\underline
{\mathcal{K}},\underline{\delta})$, can be homotopy transferred to produce an
$A_{\infty}$-algebra structure on $S^{\bullet}\underline{\mathcal{K}}$, the
symmetric algebra of $\underline{\mathcal{K}}$ (see below for details). To my
knolewdge, this remark appears here for the first time. In the second section,
I present my main framework, which consists of some basic geometry and
homological algebra of a foliation, including few not so standard aspects like
(normal) differential operators on a foliated manifold. Moreover, I define a
distinguished class of connections on a foliated manifold, that I call
\emph{adapted connections}. Finally, I use adapted connections to construct
two suitable PBW type isomorphisms $\mathcal{D}(\overline{\Lambda})\approx
S^{\bullet}\mathrm{Der}\overline{\Lambda}$ and $\overline{\Lambda}%
\otimes\overline{\mathcal{D}}\approx\overline{\Lambda}\otimes S^{\bullet
}\overline{\mathfrak{X}}$. Notice that a concept more general than an adapted
connection is used in the note \cite{lsx12} (see also \cite{csx12}) for
similar purposes, in the much wider context of \emph{Lie pairs}.
Unfortunately, \cite{lsx12} does not contain proofs. In the third section, I
collect all the constructions introduced in the preceeding sections to get the
$A_{\infty}$-algebra structure on $\overline{\Lambda}\otimes\overline
{\mathcal{D}}$ as outlined above. Finally, I compute the higher order
components of all higher operations and, in particular, prove that they vanish
from the fourth on. In the conclusions, I speculate about the interpretation
of the $A_{\infty}$-algebra $\overline{\Lambda}\otimes\overline{\mathcal{D}}$
as the \emph{universal enveloping strong homotopy algebra }of the strong
homotopy Lie-Rinehart algebra $(\overline{\Lambda},\overline{\Lambda}%
\otimes\overline{\mathfrak{X}})$.

\subsection{Conventions and notations}

I will adopt the following notations and conventions throughout the paper. Let
$k_{1},\ldots,k_{\ell}$ be positive integers. I denote by $S_{k_{1}%
,\ldots,k_{\ell}}$ the set of $(k_{1},\ldots,k_{\ell})$\emph{-unshuffles},
i.e., permutations $\sigma$ of $\{1,\ldots,k_{1}+\cdots+k_{\ell}\}$ such that
\[
\sigma(k_{1}+\cdots+k_{i-1}+1)<\cdots<\sigma(k_{1}+\cdots+k_{i-1}+k_{i}),\quad
i=1,\ldots,\ell.
\]

The degree of a homogeneous element $v$ in a graded vector space will be
denoted by $\bar{v}$. However, when it appears in the exponent of a sign
$(-)$, I will always omit the overbar, and write, for instance, $(-)^{v}$
instead of $(-)^{\bar{v}}$.

Every vector space will be over a field $K$ of zero characteristic, which will
actually be $\mathbb{R}$ in Section \ref{SecA}. If $V=\bigoplus_{i}V^{i}$ is a
graded vector space, I denote by $V[1]=\bigoplus_{i}V[1]^{i}$ its suspension,
i.e., the graded vector space defined by putting $V[1]^{i}=V^{i+1}$.

If $W$ is a (left) module over a graded, associative, graded commutative,
unital algebra $A$, I denote by $\odot$ the symmetric product in the (graded)
symmetric algebra $S_{A}^{\bullet}W$ of $W$.

Let $V_{1},\ldots,V_{n}$ be graded vector spaces,
\[
\boldsymbol{v}=(v_{1},\ldots,v_{n})\in V_{1}\times\cdots\times V_{n},
\]
and $\sigma\in S_{n}$ a permutation. I denote by $\chi(\sigma,\boldsymbol{v})$
the sign implicitly defined by
\[
v_{\sigma(1)}\wedge\cdots\wedge v_{\sigma(n)}=\chi(\sigma,\boldsymbol{v}%
)\,v_{1}\wedge\cdots\wedge v_{n},
\]
where $\wedge$ is the graded skew-symmetric product in the (graded) exterior
algebra of $V_{1}\oplus\cdots\oplus V_{n}$.

Now, let $M$ be a smooth manifold. I denote by $C^{\infty}(M)$ the real
algebra of smooth functions on $M$, by $\mathfrak{X}(M)$ the Lie-Rinehart
algebra of vector fields on $M$, and by $\Lambda(M)$ the DG algebra of
differential forms on $M$. Elements in $\mathfrak{X}(M)$ are always understood
as derivations of $C^{\infty}(M)$. Homogeneous elements in $\Lambda(M)$ are
always understood as $C^{\infty}(M)$-valued, skew-symmetric, multilinear maps
on $\mathfrak{X}(M)$. I denote by $d:\Lambda(M)\longrightarrow\Lambda(M)$ the
exterior differential. Every tensor product will be over $K$, if
not explicitly stated otherwise, and will be simply denoted by $\otimes$. {The tensor product over $C^\infty (M)$ will be denoted by $\otimes_M$}. I
adopt the Einstein summation convention.

By a \emph{connection }I will mean a linear connection in $T^{\ast}M$ or,
which is the same, in $TM$. Moreover, I will always understand the obvious
extension of a connection to the whole tensor bundle $\bigoplus_{i,j}%
TM^{\otimes i}\otimes T^{\ast}M^{\otimes j}$. Let $\nabla$ be a connection,
$\ldots,z^{a},\ldots$ coordinates in $M$, and $T$ a covariant tensor on $M$
locally given by
\[
T=T_{a_{1}\ldots a_{k}}dz^{a_{1}}\otimes\cdots\otimes dz^{a_{k}}.
\]
I denote by $\nabla_{a}T_{a_{1}\ldots a_{k}}$ the components of the covariant
derivative $\nabla T$ of $T$ with respect to $\nabla$, i.e.,
\[
\nabla T=\nabla_{a}T_{a_{1}\ldots a_{k}}dz^{a}\otimes dz^{a_{1}}\otimes
\cdots\otimes dz^{a_{k}}.
\]
Finally, the round bracket in $T_{(a_{1}\cdots a_{k})}$ denotes
symmetrization, i.e., $T_{(a_{1}\cdots a_{k})}=\frac{1}{k!}\sum_{\sigma\in
S_{k}}T_{a_{\sigma(1)}\cdots a_{\sigma(k)}}$.

\section{Algebraic Preliminaries}

\subsection{Differential Operators over Graded Commutative Algebras}

Let $A$ be an associative, graded commutative, unital $K$-algebra, and let
$P,Q$ be (left) $A$-modules. An element $a\in A$, define endomorphisms
(multiplications by $a$) $P\longrightarrow P$ and $Q\longrightarrow Q$ which,
abusing the notation, I denote again by $a$. Consider the graded $A$-linear
map
\[
\delta_{a}:\mathrm{Hom}_{K}(P,Q)\longrightarrow\mathrm{Hom}_{K}(P,Q)
\]
defined by
\[
\delta_{a}\phi:=[a,\phi]:=a\circ\phi-(-)^{a\phi}\phi\circ a,
\]
where $[\cdot,\cdot]$ is the graded commutator. A graded, $K$-linear map%
\[
\square:P\longrightarrow Q
\]
is a (linear) differential operator of order $k$ if
\[
\delta_{a_{0}}\delta_{a_{1}}\cdots\delta_{a_{k}}\square=0\quad\text{for all
}a_{0},a_{1},\ldots,a_{k}\in A.
\]

\begin{example}
A derivation of $A$ is a differential operator of order $1$. More generally, a
\emph{derivation }$\square:P\longrightarrow P$ \emph{of }$P$\emph{ subordinate
to} a derivation $\Delta$ in $A$, i.e., an operator $\square$ such that
\[
\square(ap)=\Delta(a)p+(-)^{a\square}a\square p,\quad a\in A,\quad p\in P,
\]
is a differential operator of order $1$.
\end{example}

The left $A$-module of differential operators $\square:P\longrightarrow Q$ of
order $k$ will be denoted by $\mathcal{D}_{k}(P,Q)$. Clearly, $\mathcal{D}%
_{0}(P,Q)=\mathrm{Hom}_{A}(P,Q)$ and there is a sequence of inclusions%
\[
\mathcal{D}_{0}(P,Q)\subset\cdots\subset\mathcal{D}_{k}(P,Q)\subset
\mathcal{D}_{k+1}(P,Q)\subset\cdots,
\]
defining a filtration in the $A$-module $\mathcal{D}(P,Q):=\bigcup
_{k}\mathcal{D}_{k}(P,Q)$. The associated graded object $\mathcal{S}%
(P,Q):=\bigoplus_{k}\mathcal{S}_{k}(P,Q)$, $\mathcal{S}_{k}(P,Q):=\mathcal{D}%
_{k}(P,Q)/\mathcal{D}_{k-1}(P,Q)$, is called the module of \emph{symbols}. I
denote by
\[
\sigma_{k}:\mathcal{D}_{k}(P,Q)\longrightarrow\mathcal{S}_{k}(P,Q)
\]
the projection.

Let $R$ be another $A$-module. The composition $\square_{1}\circ\square
_{2}:P\longrightarrow R$ of differential operators $\square_{1}%
:Q\longrightarrow R$ and $\square_{2}:P\longrightarrow Q$, of order $\ell_{1}$
and $\ell_{2}$, respectively, is a differential operator of order $\ell
_{1}+\ell_{2}$. Accordingly, there is a well defined $A$-bilinear map
\[
\odot{}:\mathcal{S}(Q,R)\otimes\mathcal{S}(P,Q)\longrightarrow\mathcal{S}(P,R)
\]
defined by
\[
\sigma_{\ell_{1}}(\square_{1})\odot\sigma_{\ell_{2}}(\square_{2}%
):=\sigma_{\ell_{1}+\ell_{2}}(\square_{1}\circ\square_{2}),\quad\square_{i}%
\in\mathcal{D}_{\ell_{i}}(P,Q),\quad i=1,2.
\]

I denote simply by $\mathcal{D}(A)=\bigcup_{k}$ $\mathcal{D}_{k}(A)$ (or just
$\mathcal{D}=\bigcup_{k}\mathcal{D}_{k}$, if this does not lead to confusion)
the graded, associative, filtered, unital $K$-algebra $\mathcal{D}(A,A)$ of
differential operators $A\longrightarrow A$ and by $\mathcal{S}(A)$ (or just
$\mathcal{S}$) the corresponding module of symbols. The bilinear map
$\mathcal{S}(A)\otimes\mathcal{S}(A)\longrightarrow\mathcal{S}(A)$ defined
above, gives $\mathcal{S}(A)$ the structure of an associative, graded
commutative, unital $K$-algebra. Notice that the (graded) commutator
$[\square_{1},\square_{2}]$ of differential operators $\square_{1},\square
_{2}:A\longrightarrow A$ of order $\ell_{1},\ell_{2}$, respectively, is a
differential operator of the order $\ell_{1}+\ell_{2}-1$. Accordingly, there
is a well defined $K$-bilinear bracket
\[
\{\cdot,\cdot\}:\mathcal{S}(A)\otimes\mathcal{S}(A)\longrightarrow
\mathcal{S}(A)
\]
defined by
\[
\{\sigma_{\ell_{1}}(\square_{1}),\sigma_{\ell_{2}}(\square_{2})\}:=\sigma
_{\ell_{1}+\ell_{2}-1}([\square_{1},\square_{2}]),\quad\square_{i}%
\in\mathcal{D}_{\ell_{i}},\quad i=1,2.
\]
The bracket $\{\cdot,\cdot\}$ gives $\mathcal{S}$ the structure of a graded
Poisson $K$-algebra. Notice that $\mathcal{D}_{0}=\mathcal{S}_{0}=A$,
$\mathcal{D}_{1}=A\oplus\mathrm{Der}A$ and $\mathcal{S}_{1}=\mathrm{Der}A$,
where $\mathrm{Der}A$ denotes the $A$-module of derivations of $A$.

Denote by \textrm{Der}$_{k}(A,Q)$, the $A$-module of graded symmetric,
$Q$-valued multiderivations of $A$ with $k$ entries. The map
\[
\varepsilon_{k}:\mathcal{S}_{k}(A,Q)\longrightarrow\mathrm{Der}_{k}(A,Q)
\]
given by
\[
\varepsilon_{k}\sigma_{k}(\square)(a_{1},\ldots,a_{k}):= (\delta_{a_{1}}%
\cdots\delta_{a_{k}}\square) 1,\quad\square\in\mathcal{D}_{k},\quad a_{1}%
,\ldots,a_{k}\in A
\]
is a well defined $A$-linear map.

\begin{remark}
Let $A$ be the $\mathbb{R}$-algebra of smooth functions on a graded manifold
$N$. Then $\mathrm{Der}_{k}(A,Q)\simeq Q\otimes S_{A}^{k}\mathrm{Der}A$ and
$\varepsilon_{k}$ is an isomorphism of $A$-modules, whose inverse
\[
Q\otimes S_{A}^{k}\mathrm{Der}A\longrightarrow\mathcal{S}_{k}(A,Q)
\]
si defined by
\[
q\otimes X_{1}\odot\cdots\odot X_{k}\longmapsto\sigma_{k}(qX_{1}\circ
\cdots\circ X_{k}),
\]
$q\in Q$, $X_{1},\ldots,X_{k}\in\mathfrak{X}(N)$. Moreover, $(\mathcal{S}%
,\{\cdot,\cdot\})$ is the Poisson algebra of fiber-wise polynomial functions
on $T^{\ast}N$.
\end{remark}

\subsection{Universal Enveloping of a Lie-Rinehart Algebra}

Let $A=\bigoplus_{i}A_{i}$ be an associative, graded commutative, unital
$K$-algebra, and $(A,Q)$ a graded \emph{Lie-Rinehart algebra}, i.e.,
1) $Q$ is a graded Lie algebra and 2) an $A$-module, 3) $A$ is a $Q$-module,
and 4) the following compatibility conditions hold
\begin{align*}
(a\cdot q)\cdot b  &  =a\cdot(q\cdot b)\\
q\cdot(a\cdot b)  &  =(q\cdot a)\cdot b+(-)^{aq}a\cdot(q\cdot b)\\
\lbrack a\cdot q,r]  &  =a\cdot\lbrack q,r]-(-)^{r(a+q)}(r \cdot a)\cdot q
\end{align*}
for all $a,b\in A$, $q,r\in Q$. In particular $Q$ acts on $A$ via derivations.
The prototype of a Lie-Rinehart algebra is $(A,\mathrm{Der}A)$.

An \emph{enveloping algebra} of the Lie-Rinehart algebra $(A,Q)$ is a graded,
associative, unital $K$-algebra $E$ together with 1) a morphism
$j:A\longrightarrow E$ of $K$-algebras, and 2) a morphism of Lie algebras
$J:Q\longrightarrow E$ such that 3)
\begin{align*}
J(a\cdot q)  &  =j(a)J(q)\\
j(q\cdot a)  &  =J(q)j(a)-(-)^{aq}j(a)J(q)
\end{align*}
for all $a\in A$, $q\in Q$. As an example, notice that the associative algebra
$\mathcal{D}(A)$ is an enveloping algebra of $(A,Q)$, with morphisms $j$, $J$
given by the canonical injection $A\longrightarrow\mathcal{D}(A)$ and the
action $Q\longrightarrow\mathrm{Der}A\subset\mathcal{D}(A)$.

A \emph{morphism of the enveloping algebras} $E$ and $E^{\prime}$ is a
morphism $f:E\longrightarrow E^{\prime}$ of graded, unital $K$ -algebras such
that diagrams
\[%
\begin{array}
[c]{c}%
\xymatrix@C=10pt{ E \ar[rr]^-f  & & E^\prime\\
&    A \ar[ur] \ar[ul]  & }%
\end{array}
\text{\quad and\quad}%
\begin{array}
[c]{c}%
\xymatrix@C=10pt{ E \ar[rr]^-f  & & E^\prime\\
&    Q \ar[ur] \ar[ul]  & }%
\end{array}
\]
commute. A \emph{universal enveloping algebra} is an enveloping algebra $U(Q)$
such that for any other enveloping algebra $E$ there is a unique morphism
$U(Q)\longrightarrow E$ of enveloping algebras. In particular an enveloping
algebra of $Q$ acts on $A$ by differential operators, i.e., there is a
morphism of $K$-algebras
\begin{equation}
U(Q)\longrightarrow\mathcal{D}(A). \label{15}%
\end{equation}

Universal enveloping algebras are clearly unique up to (unique) isomorphisms.
A canonical one can be constructed as follows. Let $\mathcal{U}$ be the tensor
algebra of the graded vector space $A\oplus Q$, and $I\subset\mathcal{U}$ the
two sided ideal generated by relations
\begin{align*}
a\otimes b  &  =a\cdot b\\
a\otimes q  &  =a\cdot q\\
q\otimes a-(-)^{aq}a\otimes q  &  =q\cdot a\\
q\otimes r-(-)^{qr}r\otimes q  &  =[q,r],
\end{align*}
for all $a,b\in A$, and $q,r\in Q$. Put $U(Q):=\mathcal{U}/I$. Then $U(Q)$ is
clearly a universal enveloping algebra of $Q$ with morphisms $j$, and $J$
given by the compositions of the canonical injections $A\longrightarrow
\mathcal{U}$, and $Q\longrightarrow\mathcal{U}$, with the projection
$\mathcal{U}\longrightarrow U(Q)$.

It follows from the above construction that $U(Q)$ possesses an algebra
filtration
\begin{equation}
U_{0}(Q)\subset U_{1}(Q)\subset\cdots\subset U_{i}(Q)\subset\cdots\subset U(Q)
\label{13}%
\end{equation}
bounded from below, where $U_{i}(Q)\subset U(Q)$ is the left $A$-submodule
generated by products of at most $i$ elements of the form $J(q)$, $q\in Q$. I
denote by $\mathrm{Gr}U(Q)=\bigoplus_{i}\mathrm{Gr}_{i}U(Q)$ the graded
algebra associated to the filtration (\ref{13}), i.e., $\mathrm{Gr}%
_{i}U(Q):=U_{i}(Q)/U_{i-1}(U)$. Since
\[
\lbrack U_{i}(Q),U_{j}(Q)]\subset U_{i+j-1}(Q),
\]
$\mathrm{Gr}U(Q)$ is a commutative algebra, and the commutator in $U(Q)$
induce a graded Poisson bracket in it. Notice that $U_{0}(Q)=\mathrm{Gr}%
_{0}U(Q)=A$ and $U_{1}(Q)=\mathrm{Gr}_{1}U(Q)\oplus A$ where the splitting
$U_1(Q)\longrightarrow A$ of the exact sequence
\[
0\longrightarrow A\longrightarrow U_{1}(Q)\longrightarrow\mathrm{Gr}%
_{1}U(Q)\longrightarrow0
\]
is given by
\[
\Delta \longmapsto \Delta(1).
\]

There is a canonical $A$-linear, surjective, Poisson map
\begin{equation}
S_{A}^{\bullet}Q\longrightarrow\mathrm{Gr}U(Q) \label{14}%
\end{equation}
mapping $S_{A}^{i}Q$ to $\mathrm{Gr}_{i}U(Q)$, and given by
\[
q_{1}\odot\cdots\odot q_{i}\longmapsto J(q_{1})\cdots J(q_{i})+U_{i-1}(Q).
\]

\begin{remark}
If $A$ is the graded algebra of smooth functions on a graded manifold $N$ and $Q$ is
the module of sections of a graded Lie algebroid over $N$ then $(A,Q)$ is a graded
Lie-Rinehart algebra and 1) projection (\ref{14}) is an isomorphism, moreover
2) exact sequences $0\longrightarrow U_{i-1}(Q)\longrightarrow U_{i}%
(Q)\longrightarrow\mathrm{Gr}_{i}U(Q)\longrightarrow0$ split (in a non
canonical way). Therefore there is a (non-canonical)
\emph{Poincar\'{e}-Birkhoff-Witt (PBW) type isomorphism }of (filtered)
$A$-modules
\[
U(Q)\approx S_{A}^{\bullet}Q,
\]
(for details about how to construct such isomorphism in the non-graded case
see, for instance, \cite{nwx99}). Notice that, if $(A,Q)$ is the Lie-Rinehart algebra 
of vector fields over $N$, then (\ref{15}) is an isomorphism and $U(Q)$
identifies with $\mathcal{D}(A)$ in a canonical way. Consequentely,
$\mathrm{Gr}U(Q)$ identifies with the algebra $\mathcal{S}(A)$ of symbols.
\end{remark}

Now, suppose that $A$ is a commutative DG algebra with differential $\delta$,
and $(A,Q)$ is a DG Lie-Rinehart algebra, i.e., $Q$ is endowed with a degree
$1$ differential $\delta_{0}$ such that 1) $\delta_{0}$ is a derivation of the
graded Lie algebra structure, 2) $\delta_{0}$ is a derivation of the
$A$-module $Q$ subordinate to $\delta$, i.e.,
\[
\delta_{0}(a\cdot q)=\delta a\cdot q+(-)^{a}a\cdot\delta_{0}q.
\]
In the above hypothesis, $\delta$ and $\delta_{0}$ can be extended to a unique
derivation of the tensor algebra $\mathcal{U}$. Moreover, such derivation
preserves the ideal $I$ and, therefore, descends to a derivation of $U(Q)$
which becomes a DG algebra (satisfying a DG version of the universal
properties of universal enveloping algebras) called the \emph{universal
enveloping DG algebra of the DG Lie-Rinehart algebra} $Q$.

\begin{example}
Let $A$ be the DG algebra of smooth functions on a DG manifold $N$ with
homological vector field $d$, let $Q=$\emph{ }$\mathrm{Der}A$, and let
$\delta_{0}:\mathrm{Der}A\longrightarrow\mathrm{Der}A$ be the inner derivation
$[d,\cdot]$. Then $U(Q)$ identifies with $\mathcal{D}(A)$ and the differential
in it is again $[d,\cdot]$. Hence, $\mathrm{Gr}U(Q)$ identifies with
$S_{A}^{\bullet}Q$, the DG Poisson algebra of fiberwise polynomial functions
on $T^{\ast}N$.
\end{example}

For more details about the material contained in this subsection see, for
instance, \cite{h04,mm10}.

\subsection{Strong Homotopy Structures\label{SHS}}

In this paper, conventions about strong homotopy algebras are the same as in
\cite{lm95}. Let $(V,\delta)$ be a cochain complex of vector spaces and $\mathscr{A}$
be any kind of algebraic structure (associative algebra, Lie algebra, module,
etc.). Roughly speaking, a homotopy $\mathscr{A}$-structure on $(V,\delta)$ is
an algebraic structure on $V$ which is of the kind $\mathscr{A}$ only up to
$\delta$-homotopies, and a \emph{strong homotopy (SH)} $\mathscr{A}$%
\emph{-structure} is a homotopy structure possessing a full system of
(coherent) \emph{higher homotopies}. In this paper, I will basically deal with
four kinds of SH structures, namely SH associative algebras (also named
$A_{\infty}$-algebras), SH modules (also named $A_{\infty}$-modules), SH
Lie-Rinehart algebras, and Poisson $L_\infty$-algebras. For them I provide detailed
definitions below.

\begin{definition}
\label{Def1}An $A_{\infty}$\emph{-algebra} is a pair $(\mathcal{A}%
,\mathscr{A})$, where $\mathcal{A}$ is a graded vector space, and
$\mathscr{A}=\{\alpha_{k},\ k\in\mathbb{N}\}$ is a family of $k$-ary,
multilinear, degree $2-k$ operations
\[
\alpha_{k}:{}\mathcal{A}^{\otimes k}\longrightarrow\mathcal{A},\quad
k\in\mathbb{N}.
\]
such that
\[
\sum_{i+j=k}(-)^{ij}\sum_{\ell=0}^{i+j}(-)^{\ell(i+1)+i(x_{1}+\cdots+x_{\ell
})}\alpha_{j+1}(x_{1},\ldots,x_{\ell},\alpha_{i}(x_{\ell+1},\ldots,x_{\ell
+i}),x_{\ell+i+1},\ldots,x_{i+j})=0
\]
for all $x_{1},\ldots,x_{k}\in\mathcal{A}$, $k\in\mathbb{N}$ (in
particular, $(\mathcal{A},\alpha_{1})$ is a cochain complex and $H(\mathcal{A}%
,\alpha_{1})$ is a graded associative algebra).
\end{definition}

If $\mathcal{A}$ is concentrated in degree $0$, then an $A_{\infty}$-algebra
structure on $\mathcal{A}$ is simply an associative algebra structure for
degree reasons. Similarly, if $\alpha_{k}=0$ for all $k>2$, then
$(\mathcal{A},\mathscr{A})$ is a DG (associative) algebra.

Let $(\mathcal{A},\mathscr{A})$ be an $A_{\infty}$-algebra.

\begin{definition}
A \emph{strict unit }in $\mathcal{A}$ is a degree $0$ element $e\in
\mathcal{A}$ such that $\alpha_{2}(e,x)=\alpha_{2}(x,e)=x$ for all
$x\in\mathcal{A}$ and $\alpha_{k}=0$, for all $k\neq2$, whenever one of the
entries is equal to $e$. An $A_{\infty}$-algebra with a strict unit is called
\emph{strictly unital}.
\end{definition}

Now let $M$ be a graded vector space and $\mathscr{M}=\{\mu_{k},\ k\in
\mathbb{N}\}$ a family of $k$-ary, multilinear, degree $2-k$ operations,
\[
\mu_{k}:\mathcal{A}^{\otimes(k-1)}\otimes M\longrightarrow M,\quad
k\in\mathbb{N}.
\]
Define new operations
\[
\alpha_{k}^{\oplus}:{}(\mathcal{A}\oplus M)^{\otimes k}\longrightarrow
\mathcal{A}\oplus M,\quad k\in\mathbb{N},
\]
extending the previous ones by linearity, and the condition that the result is
zero if one of the first $k-1$ entries is from $M$.

\begin{definition}
\label{Def2}An $A_{\infty}$\emph{-module} over $(\mathcal{A},\mathscr{A})$ is
a pair $(M,\mathscr{M})$, where $M$ is a graded vector space, and
$\mathscr{M}=\{\mu_{k},\ k\in\mathbb{N}\}$ is a family of $k$-ary,
multilinear, degree $2-k$ operations,
\[
\mu_{k}:\mathcal{A}^{\otimes(k-1)}\otimes M\longrightarrow M,\quad
k\in\mathbb{N},
\]
such that
\[
\sum_{i+j=k}(-)^{ij}\sum_{\ell=0}^{i+j}(-)^{\ell(i+1)+i(y_{1}+\cdots+y_{\ell
})}\alpha_{j+1}^{\oplus}(y_{1},\ldots,y_{\ell},\alpha_{i}^{\oplus}(y_{\ell
+1},\ldots,y_{\ell+i}),y_{\ell+i+1},\ldots,y_{i+j})=0
\]
for all $y_{1},\ldots,y_{k}\in\mathcal{A}\oplus M$, $k\in\mathbb{N}$ (in
particular, $(M,\mu_{1})$ is a complex and $H(M,\mu_{1})$ is a graded
$H(\mathcal{A},\alpha_{1})$-module).
\end{definition}

If both $\mathcal{A}$ and $M$ are concentrated in degree $0$, then an
$A_{\infty}$-module structure on $M$ over $\mathcal{A}$ is simply a left
module structure over the associative algebra $\mathcal{A}$. Similarly, if
$\alpha_{k}=0$ and $\mu_{k}=0$ for all $k>2$, then $(M,\mathscr{M})$ is a DG
module over the DG algebra $\mathcal{A}$.

\begin{definition}
An $L_{\infty}$\emph{-algebra} is a pair $(L,\mathscr{L})$, where $L$ is a
graded vector space, and $\mathscr{L}=\{\lambda_{k},\ k\in\mathbb{N}\}$ is a
family of $k$-ary, graded skew-symmetric, multilinear, degree $2-k$
operations
\[
\lambda_{k}:{}L^{\otimes k}\longrightarrow L,\quad k\in\mathbb{N},
\]
such that
\[
\sum_{i+j=k}(-)^{ij}\sum_{\sigma\in S_{i,j}}\chi(\sigma,\boldsymbol{v}%
)\,\lambda_{j+1}(\lambda_{i}(v_{\sigma(1)},\ldots,v_{\sigma(i)}),v_{\sigma
(i+1)},\ldots,v_{\sigma(i+j)})=0,
\]
for all $\boldsymbol{v}=(v_{1},\ldots,v_{k})$, $v_1, \ldots, v_k \in L$, $k\in
\mathbb{N}$ (in particular, $(L,\lambda_{1})$ is a complex and $H(L,\lambda
_{1})$ is a graded Lie algebra).
\end{definition}

If $L$ is concentrated in degree $0$, then an $L_{\infty}$-algebra structure
on $L$ is simply a Lie algebra structure. Similarly, if $\lambda_{k}=0$ for
all $k>2$, then $(L,\mathscr{L})$ is a DG Lie algebra.

Let $(L,\mathscr{L})$ be an $L_{\infty}$-algebra, $N$ a graded vector space,
and let $\mathscr{N}=\{\nu_{k},\ k\in\mathbb{N}\}$ be a family of $k$-ary,
graded skew-symmetric (in the first $k-1$ arguments), multilinear, degree
$2-k$ operations
\[
\nu_{k}:{}L^{\otimes(k-1)}\otimes N\longrightarrow N,\quad k\in\mathbb{N}.
\]
Define new operations
\[
\lambda_{k}^{\oplus}:{}(L\oplus N)^{\otimes k}\longrightarrow L\oplus N,\quad
k\in\mathbb{N},
\]
extending the previous ones by linearity, skew-symmetry, and the condition
that the result is zero if more than one entry are from $N$.

\begin{definition}
An $L_{\infty}$\emph{-module }is a pair $(N,\mathscr{N})$, where $N$ is a
graded vector space, and $\mathscr{N}=\{\nu_{k},\ k\in\mathbb{N}\}$ is a
family of $k$-ary, graded skew-symmetric (in the first $k-1$ arguments),
multilinear, degree $2-k$ operations
\[
\nu_{k}:{}L^{\otimes(k-1)}\otimes N\longrightarrow N,\quad k\in\mathbb{N},
\]
such that
\[
\sum_{i+j=k}(-)^{ij}\sum_{\sigma\in S_{i,j}}\chi(\sigma,\boldsymbol{b}%
)\,\lambda_{j+1}^{\oplus}(\lambda_{i}^{\oplus}(b_{\sigma(1)},\ldots
,b_{\sigma(i)}),b_{\sigma(i+1)},\ldots,b_{\sigma(i+j)})
\]
for all $\boldsymbol{b}=(v_{1},\ldots,v_{k-1},n)$, $v_1,\ldots,v_{k-1} \in L$, $n \in N$,
$k\in\mathbb{N}$ (in particular, $(N,\nu_{1})$ is a complex and $H(N,\nu_{1})$
is a graded $H(L,\lambda_{1})$-module).
\end{definition}

If both $L$ and $N$ are concentrated in degree $0$, then an $L_{\infty}%
$-module structure on $N$ over $L$ is simply a Lie module structure over the
Lie algebra $L$. Similarly, if $\lambda_{k}=0$ and $\nu_{k}=0$ for all $k>2$,
then $(N,\mathscr{N})$ is a DG Lie module over the DG Lie algebra $L$.

I now define SH Lie-Rinehart algebras \cite{k01}. For simplicity, I call the
resulting objects $LR_{\infty}$\emph{-algebras}.

\begin{definition}
\label{defLRalg}An $LR_{\infty}$\emph{-algebra }is a pair $(A,\mathcal{Q})$,
where $A$ is an associative, graded commutative, unital algebra, and
$(\mathcal{Q},\mathscr{Q})$ is an $L_{\infty}$-algebra, $\mathscr{Q}=\{\lambda
_{k},\ k\in\mathbb{N}\}$. Moreover, $\mathcal{Q}$ possesses the structure of
an $A$-module, and $A$ possesses the structure $\mathscr{M}=\{\nu_{k}%
,\ k\in\mathbb{N}\}$ of an $L_{\infty}$-module over $\mathcal{Q}$, such that

\begin{enumerate}
\item $\nu_{k}:\mathcal{Q}^{\otimes(k-1)}\otimes A\longrightarrow A$ is a
derivation in the last argument, and $A$-multilinear in the first $k-1$ arguments;

\item Formula
\begin{equation}
\lambda_{k}(q_{1},\ldots,q_{k-1},aq_{k})=\nu_{k}(q_{1},\ldots,q_{k-1}%
\,|\,a)\cdot q_{k}+(-)^{a(q_{1}+\cdots+q_{k-1}-k)}a\cdot\lambda_{k}%
(q_{1},\ldots,q_{k-1},q_{k}), \label{LRP}%
\end{equation}
holds for all $q_{1},\ldots,q_{k}\in\mathcal{Q}$, $a\in A$, $k\in\mathbb{N}$
(in particular, $(\mathcal{Q},\lambda_{1})$ is a DG module over $(A,\nu_{1})$,
and $(H(A,\nu_1),H(\mathcal{Q},\lambda_{1}))$ is a graded Lie-Rinehart algebra.
\end{enumerate}
\end{definition}

If $\mathcal{Q}$ and $A$ are concentrated in degree $0$, then $(A,Q)$ is simply a Lie-Rinehart algebra. Similarly, if $\lambda_{k}=0$ and $\nu_{k}=0$ for all $k>2$, then
$(A, Q)$ is a DG Lie-Rinehart algebra.

{In the \emph{smooth setting}, i.e., when $A$ is the algebra of smooth
functions on a smooth manifold }$M${ (in particular }$A$ is concentrated in
degree $0${), and $\mathcal{Q}[-1]$ is the }$A$-module of sections of a graded
bundle $\mathcal{E}$ over {$M$}, then $\mathcal{E}${ is sometimes
called an $L_{\infty}$-algebroid \cite{sss09,b11,bp12}.}

\begin{definition}
A \emph{Poisson $L_\infty$-algebra} is an $L_{\infty}$-algebra $(\mathcal{P}%
,\mathscr{P})$, $\mathscr{P}=\{\Lambda_{k},\ k\in\mathbb{N}\}$, such that
$\mathcal{P}$ possesses the structure of an associative, graded commutative,
unital algebra and $\Lambda_{k}$ is a graded multiderivation for all
$k\in\mathbb{N}$.
\end{definition}

{\begin{remark}
Poisson $L_\infty$-algebras are called $P_\infty$-algebras in \cite{cf07}. Notice that they are homotopy versions of Poisson algebra where \textquotedblleft only the Poisson bracket is homotopyfied\textquotedblright, while the associative, commutative product is not. More general versions of Poisson algebras up to homotopy can be obtained via the (systematic) operadic approach to homotopy algebras (see, for instance, \cite{v99}). This is the main reason why, as suggested by an anonymous referee, I do not use the name \emph{SH Poisson algebras} for Poisson $L_\infty$-algebras. Similar considerations hold actually for Definition \ref{defLRalg} where \textquotedblleft only the Lie bracket and Lie module structure on a Lie-Rinehart algebra are homotopyfied\textquotedblright\ while the associative, commutative product, and the corresponding module structure are not. In this case, however, it is safer to keep the name \emph{SH Lie-Rinehart algebra} since an operadic approach in this context is still missing. 
\end{remark}}

Notice that if $\mathcal{P}$ is concentrated in degree $0$, then a SH Poisson
algebra structure on $\mathcal{P}$ is simply a Poisson algebra structure.
Similarly, if $\Lambda_{k}=0$ for all $k>2$, then $(\mathcal{P},\mathscr{P})$
is a DG Poisson algebra.

\begin{remark}
\label{RemLRP}Let $A$ be an associative, graded commutative, unital algebra
and $\mathcal{Q}$ an $A$-module. The datum of an $LR_{\infty}$-algebra
structure on $(A,\mathcal{Q})$ is equivalent to the datum of a SH Poisson
algebra structure on $S_{A}^{\bullet}\mathcal{Q}$ such that
\[
\Lambda_{k}(u_{1},\ldots,u_{k})\in S_{A}^{p_{1}+\cdots+p_{k}-k+1}\mathcal{Q}%
\]
whenever $u_{i}\in S_{A}^{p_{i}}\mathcal{Q}$, $i=1,\ldots,k$ \cite{b11}. The
operations in $S_{A}^{\bullet}\mathcal{Q}$ can be obtained from the ones in
$\mathcal{Q}$, extending the latter as multiderivations.
\end{remark}

Finally, notice that the canonical construction of a Lie algebra from an
associative algebra can be generalized to the SH context as follows. Let
$(\mathcal{A},\mathscr{A})$ be an $A_{\infty}$-algebra, $\mathscr{A}=\{\alpha
_{k},\ k\in\mathbb{N}\}$. Define new operations
\[
\mathsf{A}\alpha_{k}:\mathcal{A}^{\otimes k}\longrightarrow\mathcal{A},
\]
by putting
\[
(\mathsf{A}\alpha_{k})(x_{1},\ldots,x_{k}):=\sum_{\sigma\in S_{k}}\chi
(\sigma,\boldsymbol{x})\alpha_{k}(x_{\sigma(1)},\ldots,x_{\sigma(k)}),
\]
$\boldsymbol{x}=(x_{1},\ldots,x_{k})$, $x_1, \ldots, x_k \in\mathcal{A}$, i.e.,
$\mathsf{A}\alpha_{k}$ is the \emph{skew-symmetrization} of $\alpha_{k}$. The
$\mathsf{A}\alpha$'s give to $\mathcal{A}$ the structure of an $L_{\infty}%
$-algebra \cite{ls93}.

\begin{remark}
The theory of universal enveloping of $L_{\infty}$-algebras (see, for
instance, \cite{b08}) is not fully developed, not to speak about universal
enveloping of $LR_{\infty}$-algebras. However, few (naive) remarks can be done
in this respect. First of all, recall that a morphism $f:\mathcal{A}%
\longrightarrow\mathcal{A}^{\prime}$ (resp., $f:L\longrightarrow L^{\prime}$)
of $A_{\infty}$-algebras (resp., $L_{\infty}$-algebras) is a family of
$K$-multilinear (resp., skew-symmetric, $K$-multilinear) maps $f_{k}%
:\mathcal{A}^{\otimes k}\longrightarrow\mathcal{A}^{\prime}$ (resp.,
$f_{k}:L^{\otimes k}\longrightarrow L^{\prime})$ satisfying suitable
compatibility conditions (see, for instance, \cite{lm95} for details). It is
tempting to define an \emph{enveloping SH algebra} for an $LR_{\infty}%
$-algebra $\mathcal{Q}$ over a DG algebra $A$, as an $A_{\infty}$-algebra
$\mathcal{E}$ together with 1) a morphism of DG algebras $j:A\longrightarrow
\mathcal{E}$, and 2) a morphism of $L_{\infty}$-algebras $J:Q\longrightarrow
\mathcal{E}$ such that 3)
\[
J_{k}(a\cdot q_{1},q_{2}\ldots,q_{k})=j(a)J_{k}(q_{1},\ldots q_{k})
\]
and
\begin{align}
&  j(\nu_{k}(q_{1},\ldots,q_{k-1}|a))\nonumber\\
&  =\quad\ \sum_{\ell=1}^{k-1}\sum_{\substack{k_{1}+\cdots+k_{\ell}%
=k-1\\k_{1}\leq\cdots\leq k_{\ell}}}\sum_{\sigma\in S_{k_{1},\ldots,k_{\ell}%
}^{<}}\chi(\sigma,\boldsymbol{q})(\mathsf{A}\alpha_{\ell+1})(J_{k_{1}%
}(q_{\sigma(1)},\ldots),\ldots,J_{k_{\ell}}(\ldots,q_{\sigma(k-1)}),ja),
\label{23}%
\end{align}
(here $S_{k_{1},\ldots,k_{\ell}}^{<}\subset S_{k_{1},\ldots,k_{\ell}}$ is the
set of $(k_{1},\ldots,k_{\ell})$-unshuffles such that
\[
\sigma(k_{1}+\cdots+k_{i-1}+1)<\sigma(k_{1}+\cdots+k_{i-1}+k_{i}+1)\text{\quad
whenever }k_{i}=k_{i+1}\text{,}%
\]
see the definition of morphism of $L_{\infty}$-algebras, e.g., in
\cite{lm95}). One could then define a \emph{universal enveloping SH algebra}
as an enveloping SH algebra satisfying (obvious) universal properties, and try
to construct it. Developing these ideas, however, goes beyond the scopes of this paper.
\end{remark}

\subsection{Homological Perturbations and Homotopy Transfer\label{SecHPHT}}

The main homological tools used in this paper are the \emph{Perturbation
Lemma} and the \emph{Homotopy Transfer Theorem}. I recall in this section
those versions of them that will be used below.

Let $(\mathcal{K},\delta)$ and $(\underline{\mathcal{K}},\underline{\delta})$
be cochain complexes of vector spaces, $p:(\mathcal{K},\delta)\longrightarrow
(\underline{\mathcal{K}},\underline{\delta})$ and $j:(\underline{\mathcal{K}%
},\underline{\delta})\longrightarrow(\mathcal{K},\delta)$ cochain maps, and
let $h:\mathcal{K}\longrightarrow\mathcal{K}$ be a degree $-1$ endomorphism:
\[
\xymatrix{     *{ \quad \quad(\mathcal{K}, \delta)\ }
\ar@(dl,ul)[]^-{h}\
\ar@<0.5ex>[r]^-{p} & *{\
(\underline{\mathcal{K}},\underline{\delta})\quad\ \  \ \quad}  \ar@<0.5ex>[l]^-{j}}
\]

\begin{definition}
The data $(p,j,h)$ are \emph{contraction data for} $(\mathcal{K},\delta)$
\emph{over} $(\underline{\mathcal{K}},\underline{\delta})$ if

\begin{enumerate}
\item $j$ is a right inverse of $p$, i.e., $pj=\mathrm{id}$,

\item $h$ is a contracting homotopy, i.e., $[h,\delta]=\mathrm{id}-jp$,

\item the \emph{side conditions} $h^{2}=0$, $hj=0$, $ph=0$ are satisfied.
\end{enumerate}
\end{definition}

Now, let $(p_{0},j_{0},h_{0})$ be contraction data for a cochain complex
$(\mathcal{K},\delta_{0})$ over $(\underline{\mathcal{K}},\underline{\delta}%
)$. Suppose that there is another differential $\delta$ in $\mathcal{K}$, and
put $t:=\delta_{0}-\delta$. The Perturbation Lemma allows one to construct
contraction data for $(\mathcal{K},\delta)$ over a suitable new complex
$(\underline{\mathcal{K}},\underline{\delta}{}_{t})$.

\begin{theorem}
[Perturbation Lemma]Let $th_{0}:\mathcal{K}\longrightarrow\mathcal{K}$ be
locally nilpotent, i.e., for any $x\in\mathcal{K}$ there is $k\in\mathbb{N}$
such that $(th_{0})^{k}(x)=0$, and
\[
X:=t+th_{0}t+th_{0}th_{0}t+\cdots=\sum_{i=0}^{\infty}t(h_{0}t)^{i}=\sum
_{i=0}^{\infty}(th_{0})^{i}t.
\]
Moreover, let $\underline{\delta}{}_{t},p_{t},j_{t},h_{t}$ be defined as%
\begin{align}
\underline{\delta}{}_{t}  &  :=\underline{\delta}-p_{0}Xj_{0}\nonumber\\
p_{t}  &  :=p_{0}(\mathrm{id}+Xh_{0})\label{p}\\
j_{t}  &  :=(\mathrm{id}+h_{0}X)j_{0}\label{j}\\
h_{t}  &  :=h_{0}+h_{0}Xh_{0}. \label{hi}%
\end{align}
Then $(p_{t},j_{t},h_{t})$ are contraction data for $(\mathcal{K},\delta)$
over $(\underline{\mathcal{K}},\underline{\delta}{}_{t})$.
\end{theorem}

\begin{remark}
\label{Rem1} A rather standard situation, which will be also encountered in
this paper, is when $\mathcal{K}$ and $\underline{\mathcal{K}}$ are endowed
with filtrations
\begin{align*}
\mathcal{K}_{0}  &  \subset\mathcal{K}_{1}\subset\cdots\subset\mathcal{K}%
_{i}\subset\cdots\subset\mathcal{K},\\
\underline{\mathcal{K}}{}_{0}  &  \subset\underline{\mathcal{K}}{}_{1}%
\subset\cdots\subset\underline{\mathcal{K}}{}_{i}\subset\cdots\subset
\underline{\mathcal{K}},
\end{align*}
bounded from below, and such that, 1) they are preserved by $\delta
_{0},\underline{\delta},p_{0},j_{0},h_{0}$, and 2) $t(\mathcal{K}_{i}%
)\subset\mathcal{K}_{i-1}$. In this case, $th_{0}$ is automatically locally
nilpotent and the Perturbation Lemma applies.
\end{remark}

Contraction data for $(\mathcal{K},\delta)$ over $(\underline{\mathcal{K}%
},\underline{\delta})$ can be used to transfer SH structures from the former
to the latter, in particular when the SH structure one begins with does not
possess higher homotopies. This is a rather rich source of SH structures.
Moreover, there are explicit formulas for the higher homotopies of the induced structure.

{
\begin{theorem}
[Homotopy Transfer Theorem, see, e.g., \cite{m99,h10}]\label{HTT}Let
$(\mathcal{V},\delta)$ and $(\underline{\mathcal{V}},\underline{\delta})$ be
cochain complexes and let $(p,j,h)$ be contraction data for $(\mathcal{V}%
,\delta)$ over $(\underline{\mathcal{V}},\underline{\delta})$.
\begin{enumerate}
\item Assume $(\mathcal{V},\delta)$ possesses the structure $\circ$ of a DG
algebra, and let $\mathscr{A}=\{\alpha_{k},\ k\in\mathbb{N}\}$ be the family
of graded operations
\[
\alpha_{k}:\underline{\mathcal{V}}^{\otimes k}\longrightarrow\underline
{\mathcal{V}}%
\]
defined by
\[
\alpha_{1}:=\underline{\delta},\quad\alpha_{k}:=p\beta_{k},\quad k\geq2,
\]
where the $\beta$'s are inductively defined by
\[
\gamma_{1}:=-j,\quad\gamma_{k}:=h\beta_{k},
\]
and
\[
\beta_{k}(x_{1},\ldots,x_{k}):=\sum_{\ell+m=k}(-)^{a(\ell,m,\boldsymbol{x}%
)}\gamma_{\ell}(x_{1},\ldots,x_{\ell})\circ\gamma_{m}(x_{\ell+1}%
,\ldots,x_{\ell+m}),
\]
$x_{1},\ldots,x_{k}\in\underline{\mathcal{V}}$, where $a(\ell,m,\boldsymbol{x}%
):=\ell-1+(m-1)%
{\textstyle\sum_{i=1}^{\ell}}
\bar{x}_{i}$, $k\geq2$. Then $(\underline{\mathcal{V}},\mathscr{A})$ is an
$A_{\infty}$-algebra. Moreover, if $(\mathcal{V},\delta)$ is a unital DG
algebra with unit $1_{\mathcal{V}}$ such that $(jp)1_{\mathcal{V}%
}=1_{\mathcal{V}}$, then $(\underline{\mathcal{V}},\mathscr{A})$ is a strictly
unital $A_{\infty}$-algebra with unit $p1_{\mathcal{V}}$.
\item Assume $(\mathcal{V},\delta)$ possesses the structure $[{}\cdot{}%
,{}\cdot{}]$ of a DG Lie algebra, and let $\mathscr{L}=\{\lambda_{k}%
,\ k\in\mathbb{N}\}$ be the family of graded operations
\[
\lambda_{k}:\underline{\mathcal{V}}^{\otimes k}\longrightarrow\underline
{\mathcal{V}}%
\]
defined by
\begin{equation}
\lambda_{1}:=\underline{\delta},\quad\lambda_{k}:=p\phi_{k},\quad
k\geq2\label{For}%
\end{equation}
where the $\phi$'s are inductively defined by
\[
\psi_{1}:=-j,\quad\psi_{k}:=h\phi_{k},
\]
and
\begin{align*}
&  \phi_{k}(x_{1},\ldots,x_{k})\\
&  :=\sum_{\ell+m=k}\sum_{\sigma\in S_{\ell,m}}(-)^{b(\ell,m,\boldsymbol{x}%
)}\chi(\sigma,\boldsymbol{x})[\psi_{\ell}(x_{\sigma(1)},\ldots,x_{\sigma
(\ell)}),\psi_{m}(x_{\sigma(\ell+1)},\ldots,x_{\sigma(\ell+m)})],
\end{align*}
$x_{1},\ldots,x_{k}\in\underline{\mathcal{V}}$, where $b(\ell,m,\boldsymbol{x}%
):=\ell-1+(m-1)\sum_{i=1}^{\ell}x_{\sigma(i)}$, $k\geq2$. Then
$(\underline{\mathcal{V}},\mathscr{L})$ is an $L_{\infty}$-algebra.
\end{enumerate}
\end{theorem}}

\subsection{Homotopy Transfer of Universal Enveloping\label{SecHTUE}}

In this subsection I present an abstract algebraic model for the concrete
geometric framework of the next section.

SH module structures can be transferred along contraction data similarly as in
the previous subsection. Even more, one can transfer a SH Lie-Rinehart algebra
structure along suitable contraction data. Namely, let $(A,\delta)$ be a
commutative, unital DG algebra, let $\mathcal{K}$ be a DG Lie-Rinehart algebra
over $(A,\delta)$ with differential $\delta_{0}$ and Lie-bracket $[{}\cdot
{},{}\cdot{}]$, and let $\underline{\mathcal{K}}$ be a DG-module over
$(A,\delta)$ with differential $\underline{\delta}$. Moreover, suppose that
there are $A$-linear contraction data $(p_{0},j_{0},h_{0})$ for $(\mathcal{K}%
,\delta_{0})$ over $(\underline{\mathcal{K}},\underline{\delta})$. Then, it is
easy to see that there is an $LR_{\infty}$-algebra structure $\mathscr{Q}$ in
$\underline{\mathcal{K}}$ defined in a similar way as in Theorem \ref{HTT}. I
do not report here the obvious details.

Now, consider the symmetric DG algebras $S_{A}^{\bullet}\mathcal{K}$ and
$S_{A}^{\bullet}\underline{\mathcal{K}}$. In view of Remark \ref{RemLRP}, they
are endowed with a DG Poisson structure and a Poisson $L_\infty$-algebra structure
$\mathscr{P}$, respectively. I denote 1) by $\{{}\cdot{},{}\cdot{}\}$ the
Poisson bracket in $S_{A}^{\bullet}\mathcal{K}$, and 2) again by $\delta_{0}$
and $\underline{\delta}$ the differentials in $S_{A}^{\bullet}\mathcal{K}$ and
$S_{A}^{\bullet}\underline{\mathcal{K}}$, respectively. I claim that the
contraction data $(p_{0},j_{0},h_{0})$ extend to contraction data%
\[
\xymatrix{     *{ \quad \quad(S_A^\bullet \mathcal{K}, \delta_0)\ }
\ar@
(dl,ul)[]+<-3ex,-1ex>;[]+<-3ex,1ex>^-{h_0}
\ar@<0.5ex>[r]^-{p_0} & *{\
(S_A^\bullet\underline{\mathcal{K}},\underline{\delta})\quad\ \  \ \quad}  \ar@<0.5ex>[l]^-{j_0}}
\]
such that the above mentioned Poisson $L_\infty$-algebra structure on $S_{A}^{\bullet
}\underline{\mathcal{K}}$ is obtained from the DG Poisson structure on
$S_{A}^{\bullet}\mathcal{K}$ via homotopy transfer. Indeed, put $\mathcal{Z}%
:=\ker p_{0}$. Then $\mathcal{K}=\underline{\mathcal{K}}\oplus\mathcal{Z}$ and
$S_{A}^{\bullet}\mathcal{K}\simeq S_{A}^{\bullet}\underline{\mathcal{K}%
}\otimes_A S_{A}^{\bullet}\mathcal{Z}$. Now, extend $p_{0}$ and $j_{0}$ as
algebra morphisms, and let $h^{\prime}:S_{A}^{\bullet}\mathcal{K}%
\longrightarrow S_{A}^{\bullet}\mathcal{K}$ be the extension of $h_{0}$ as a
derivation. For $\Sigma\in S_{A}^{\bullet}\underline{\mathcal{K}}\otimes
S_{A}^{i}\mathcal{Z}\subset S_{A}^{\bullet}\mathcal{K}$, put
\[
h_{0}\Sigma:=\left\{
\begin{array}
[c]{cc}%
0 & \text{if }i=0\\
\frac{1}{i}h^{\prime}\Sigma & \text{if }i>0
\end{array}
\right.  .
\]
It is easy to see that $(p_{0},j_{0},h_{0})$ are contraction data for
$(S_{A}^{\bullet}\mathcal{K},\delta_{0})$ over $(S_{A}^{\bullet}%
\underline{\mathcal{K}},\underline{\delta})$ extending the previous ones.
Thus, in view of the Homotopy Transfer Theorem, there is an $L_{\infty}%
$-algebra structure $\mathscr{L}=\{\lambda_{k},\ k\in\mathbb{N}\}$ in
$S_{A}^{\bullet}\underline{\mathcal{K}}$ given by Formulas (\ref{For}). Notice
that $h_{0}:S_{A}^{\bullet}\mathcal{K}\longrightarrow S_{A}^{\bullet
}\mathcal{K}$ is $S_{A}^{\bullet}\underline{\mathcal{K}}$-linear, i.e.,
\[
h_{0}(j_{0}\Sigma\odot\Sigma^{\prime})=(-)^{\Sigma}j_{0}\Sigma\odot
h_{0}\Sigma^{\prime},\quad\text{for all }\Sigma\in S_{A}^{\bullet}%
\underline{\mathcal{K}}\text{, }\Sigma^{\prime}\in S_{A}^{\bullet}\mathcal{K}.
\]

\begin{proposition}
The structures $\mathscr{P}$ and $\mathscr{L}$ coincide.
\end{proposition}

\begin{proof}
Since $\mathscr{L}$ extends $\mathscr{Q}$, it is enough to show that the
$\lambda$'s are multiderivations. This can be proved by induction as follows.
I claim that, for any $k$, $\phi_{k}$ is an \textquotedblleft
approximate\textquotedblright\ multiderivation along $j_{0}$ in the following
sense:
\begin{align}
&  \phi_{k}(\Sigma^{\prime}\odot\Sigma^{\prime\prime},\Sigma_{1},\ldots
,\Sigma_{k-1})\nonumber\\
&  =(-)^{\Sigma^{\prime}k}j_{0}\Sigma^{\prime}\odot\phi_{k}(\Sigma
^{\prime\prime},\Sigma_{1},\ldots,\Sigma_{k-1})+(-)^{\Sigma^{\prime\prime
}(\Sigma_{1}+\cdots+\Sigma_{k-1})}\phi_{k}(\Sigma^{\prime},\Sigma_{1}%
,\ldots,\Sigma_{k-1})\odot j_{0}\Sigma^{\prime\prime}+I \label{EqI}
\end{align}
for all $\Sigma^{\prime},\Sigma^{\prime\prime},\Sigma_{1},\ldots,\Sigma
_{k-1}\in S_{A}^{\bullet}\underline{\mathcal{K}}$, where $I$ is the ideal of
$(S_{A}^{\bullet}\mathcal{K},\odot)$ generated by the image of $h_{0}$. Since
$I\subset\ker p_{0}$, it follows from the claim and the side condition
$p_{0}h_{0}=0$ that $\lambda_{k}$ is a multiderivation. Now, prove the claim
by induction on $k$. First of all, a straightforward computation shows that
\[
\phi_{2}(\Sigma^{\prime}\odot\Sigma^{\prime\prime},\Sigma)=j_{0}\Sigma
^{\prime}\odot\phi_{2}(\Sigma^{\prime\prime},\Sigma)+(-)^{\Sigma\Sigma
^{\prime\prime}}\phi_{2}(\Sigma^{\prime},\Sigma)\odot j_{0}\Sigma
^{\prime\prime}.
\]
Now, assume that (\ref{EqI}) holds for all $k\leq n$, and prove it for $k=n+1$.
From skew-symmetry it is enough to check it on equal, odd elements $\Sigma
_{1}=\cdots=\Sigma_{n}=\Sigma$. Put $\boldsymbol{\Sigma}:=(\Sigma^{\prime
}\odot\Sigma^{\prime\prime},\Sigma^{n})$ and compute
\begin{align*}
&  \phi_{n+1}(\Sigma^{\prime}\odot\Sigma^{\prime\prime},\Sigma^{n})\\
&  =2\sum\nolimits_{\ell+m=n}(-)^{b(\ell,m,\boldsymbol{\Sigma})}\tbinom
{\ell+m}{\ell}\{\psi_{\ell+1}(\Sigma^{\prime}\odot\Sigma^{\prime\prime}%
,\Sigma^{\ell}),\psi_{m}(\Sigma^{m})\}\\
&  =-2\{j_{0}\Sigma^{\prime}\odot j_{0}\Sigma^{\prime\prime},h_{0}\phi
_{n}(\Sigma^{n})\}\\
&  \quad\ +2\sum\nolimits_{\ell=1}^{n-1}(-)^{b(\ell,n-\ell,\boldsymbol{\Sigma
})}\tbinom{n}{\ell}\{h_{0}\phi_{\ell+1}(\Sigma^{\prime}\odot\Sigma
^{\prime\prime},\Sigma^{\ell}),h_{0}\phi_{n-\ell}(\Sigma^{n-\ell})\}+I\\
&  =-2j_{0}\Sigma^{\prime}\odot\{j_{0}\Sigma^{\prime\prime},h_{0}\phi
_{n}(\Sigma^{n})\}-2(-)^{\Sigma^{\prime\prime}}\{j_{0}\Sigma^{\prime\prime
},h_{0}\phi_{n}(\Sigma^{n})\}\odot j_{0}\Sigma^{\prime\prime}\\
&  \quad\ +2\sum\nolimits_{\ell=1}^{n-1}(-)^{b(\ell,n-\ell,\boldsymbol{\Sigma
})}\tbinom{n}{\ell}[(-)^{\ell\Sigma^{\prime}}\{j_{0}\Sigma^{\prime}\odot
h_{0}\phi_{\ell+1}(\Sigma^{\prime\prime},\Sigma^{\ell}),h_{0}\phi_{n-\ell
}(\Sigma^{n-\ell})\}\\
&  \quad\ +(-)^{\ell\Sigma^{\prime\prime}}\{h_{0}\phi_{\ell+1}(\Sigma^{\prime
},\Sigma^{\ell})\odot j_{0}\Sigma^{\prime\prime},h_{0}\phi_{n-\ell}%
(\Sigma^{n-\ell})\}]+I\\
&  =(-)^{\Sigma^{\prime}(n+1)}j_{0}\Sigma^{\prime}\odot\phi_{n+1}%
(\Sigma^{\prime\prime},\Sigma^{n})+(-)^{n\Sigma^{\prime\prime}}\phi
_{n+1}(\Sigma^{\prime},\Sigma^{n})\odot j_{0}\Sigma^{\prime\prime}+I
\end{align*}
where I used the fact that, since $h_{0}(I)\subset I\odot I$, then
$\{h_{0}(I),S_{A}^{\bullet}\mathcal{K}\}\subset I$.
\end{proof}

Now, let $(U(\mathcal{K}),\delta)$ be the universal enveloping DG algebra of
$(\mathcal{K},\delta_{0})$, and suppose there is a PBW type isomorphism
$U(\mathcal{K})\approx S_{A}^{\bullet}\mathcal{K}$, i.e., an isomorphism
$\mathsf{PBW}:S_{A}^{\bullet}\mathcal{K}\longrightarrow U(\mathcal{K})$ of
filtered $A$-modules such that diagram
\[%
\xymatrix@C=10pt{ S_A^{\leq k} \mathcal{K}
\ar[rr]^-{\mathsf{PBW}} \ar[dr]  & & U_k(\mathcal{K}) \ar[dl] \\
&          \mathrm{Gr}_k U(\mathcal{K}) &}%
\]
commutes for all $k$, (the map $S_{A}^{\leq k}\mathcal{K}\longrightarrow
\mathrm{Gr}_{k}U(\mathcal{K})$ being the composition of projections
$S_{A}^{\leq k}\mathcal{K}\longrightarrow S_{A}^{k}\mathcal{K}$ and $S_{A}%
^{k}\mathcal{K}\longrightarrow\mathrm{Gr}_{k}U(\mathcal{K})$). Use
$\mathsf{PBW}$ to identify $U(\mathcal{K})$ and $S_{A}^{\bullet}\mathcal{K}$.
Then 1) the filtrations in $U(\mathcal{K})=S_{A}^{\bullet}\mathcal{K}$ and
$S_{A}^{\bullet}\underline{\mathcal{K}}$ are preserved by $\delta
_{0},\underline{\delta},p_{0},j_{0},h_{0}$, and 2) $t(U_{i}(\mathcal{K}%
))\subset U_{i-1}(\mathcal{K})$. It follows from Remark \ref{Rem1} and the
Perturbation Lemma that there are contraction data $(p_{t},j_{t},h_{t})$ for
$(U(\mathcal{K}),\delta)$ over $(S_{A}^{\bullet}\underline{\mathcal{K}%
},\underline{\delta}{}_{t})$. Hence, in view of the Homotopy Transfer Theorem,
there is an $A_{\infty}$-algebra structure on $S_{A}^{\bullet}\underline
{\mathcal{K}}$ canonically determined by the contraction data $(p_{0}%
,j_{0},h_{0})$ and the isomorphism $\mathsf{PBW}$.

{
\begin{remark}
The $A_\infty$-algebra structure (induced as above) on $S^\bullet_A \underline{\mathcal{K}}$ highly depends on the isomorphism $\mathsf{PBW}$ (besides the contraction data), and it could be hard to write explicit formulas in practice. In the case of the $A_\infty$-algebra of a foliation, I will only compute the highest order contributions to the first few homotopies (see Section \ref{SecA} for details).
\end{remark}
}

\begin{example}
\label{Examp}Let $M$ be a smooth manifold, $\mathcal{F}$ a foliation of $M$,
and $C$ its characteristic distribution. Moreover, let $(\mathcal{K}%
,\delta_{0})$ be the deformation complex of $\mathcal{F}$ \cite{cm08} and
$(\underline{\mathcal{K}},\underline{\delta})$ the Chevalley-Eilenberg complex
determined by the Bott connection in $TM/C$ (see Section \ref{SecDGHAF} for
more details). A splitting $TM=C\oplus V$ via a complementary distribution $V$
determines contraction data $(p_{0},j_{0},h_{0})$ for $(\mathcal{K},\delta
_{0})$ over $(\underline{\mathcal{K}},\underline{\delta})$. Accordingly, there
is an $LR_{\infty}$-algebra structure on $\underline{\mathcal{K}}$ which I
described in \cite{vi12} (see also \cite{h05,j12}). In Subsection \ref{PBW} I
show how to construct a PBW isomorphism $U(\mathcal{K})\approx S^{\bullet
}\mathcal{K}$, via purely geometric data (specifically, a connection). One
immediately concludes that there is an $A_{\infty}$-algebra structure on
$S^{\bullet}\underline{\mathcal{K}}$. I partially describe this $A_{\infty}%
$-algebra in Section \ref{SecA}. Here, I present the toy example when
$\mathcal{F}$ has just one leaf and $C=TM$, as an illustration of the main
technical aspects of the general case.

When $C=TM$, the deformation complex of $\mathcal{F}$ is $(\mathrm{Der}%
\Lambda(M),\delta_{0}=[d,\cdot])$, $TM/C=0$, and its Chevalley-Eilenberg
complex $(\underline{\mathcal{K}},\underline{\delta})$ is trivial. Put
$\Lambda:=\Lambda(M)$. There are contraction data $(0,0,h_{0})$ for
$(\mathrm{Der}\Lambda,\delta_{0})$ over the $0$ complex. The contracting
homotopy $h_{0}$ is defined as follows. Every element $\Delta\in
\mathrm{Der}\Lambda$ can be uniquely written as \cite{m08} $\Delta=i_{U}%
+L_{V}$, $U,V\in\Lambda\otimes_M\mathfrak{X}(M)$. Then $h_{0}(\Delta
):=(-)^{\Delta}i_{V}$. The homotopy $h_{0}$ is $\Lambda$-linear. Accordingly,
it determines contraction data $(p_{0},j_{0},h_{0})$ for $(\mathcal{S}%
(\Lambda)=S_{\Lambda}^{\bullet}\mathrm{Der}\Lambda,\delta_{0})$ over
$(S_{\Lambda}^{\bullet}\underline{\mathcal{K}},\underline{\delta}%
)=(\Lambda,d)$, where $p_{0}:S_{\Lambda}^{\bullet}\mathrm{Der}\Lambda
\longrightarrow\Lambda$ and $j_{0}:\Lambda\longrightarrow S_{\Lambda}%
^{\bullet}\mathrm{Der}\Lambda$ are the obvious maps. Notice that the SH
Poisson algebra structure induced on $(\Lambda,d)$ is trivial. The universal
enveloping DG algebra of $\mathrm{Der}\Lambda$ is $(\mathcal{D}(\Lambda
),\delta=[d,{}\cdot{}])$. A PBW isomorphism $\mathcal{D}(\Lambda
)\approx\mathcal{S}(\Lambda)$ can be constructed, exploiting a connection
$\nabla$, as follows (see \cite{f63,nwx99} for similar results). Extend the
covariant derivative $\nabla:\mathfrak{X}(M)\longrightarrow\mathrm{Der}%
\Lambda$ to the whole $\Lambda\otimes_M\mathfrak{X}(M)$ by $\Lambda$-linearity.
For $Z\in\Lambda\otimes_M\mathfrak{X}(M)$, $L_{Z}-\nabla_{Z}=i_{\nabla Z}$. It
follows that every element $\Delta$ in $\mathrm{Der}\Lambda$ can be uniquely
written in the form $\Delta=i_{U}+\nabla_{Z}$, $U,Z\in\Lambda\otimes_M
\mathfrak{X}(M)$, and the correspondence
\[
\Lambda\otimes_M\mathfrak{X}(M)[1]\oplus\Lambda\otimes_M\mathfrak{X}%
(M)\ni(U,Z)\longmapsto i_{U}+\nabla_{Z}\in\mathrm{Der}\Lambda
\]
is a well defined isomorphism of $\Lambda$-modules. Accordingly,
$\mathcal{S}(\Lambda)$ identifies with
\[
S_{\Lambda}^{\bullet}(\Lambda\otimes_M\mathfrak{X}(M)[1])\underset{\Lambda
}{\otimes}S_{\Lambda}^{\bullet}(\Lambda\otimes_M\mathfrak{X}(M))\simeq
\Lambda\otimes_M\Lambda^{\bullet}\mathfrak{X}(M)\otimes_M S^{\bullet}%
\mathfrak{X}(M).
\]
Now, let
\[
\Sigma=\omega\otimes Y_{1}\wedge\cdots\wedge Y_{j}\otimes P\in\Lambda
\otimes_M\Lambda^{j}\mathfrak{X}(M)\otimes_M S^{\ell}\mathfrak{X}(M),
\]
let $\ldots,z^{a},\ldots$ be coordinates in $M$, and let $P$ be locally given
by $P=P^{a_{1}\cdots a_{\ell}}\frac{\partial}{\partial z^{a_{1}}}\odot
\cdots\odot\frac{\partial}{\partial z^{a_{\ell}}}$. Define $\nabla_{P}%
:\Lambda\longrightarrow\Lambda$ via local formulas$\ \nabla_{P}:=P^{a_{1}%
\cdots a_{\ell}}\nabla_{a_{1}}\cdots\nabla_{a_{\ell}}$, and put
\begin{equation}
\mathsf{PBW}(\Sigma):=\omega i_{Y_{1}}\cdots i_{Y_{j}}\nabla_{P}\in
\mathcal{D}_{j+\ell}(\Lambda).\label{5''}%
\end{equation}
The restrictions $\mathsf{PBW}:\mathcal{S}_{i}(\Lambda)\longrightarrow
\mathcal{D}_{i}(\Lambda)$ split the exact sequences $0\longrightarrow
\mathcal{D}_{i-1}(\Lambda)\longrightarrow\mathcal{D}_{i}(\Lambda
)\longrightarrow\mathcal{S}_{i}(\Lambda)\longrightarrow0$, so that
$\mathsf{PBW}$ is the required PBW isomorphism. The Perturbation Lemma gives
now contraction data for $(\mathcal{D}(\Lambda),\delta)$ over $(\Lambda,d)$.
The $A_{\infty}$-algebra structure induced on $(\Lambda,d)$ is again trivial.
\end{example}

\section{Geometric Preliminaries}

\subsection{(A Bit of) Differential Geometry and Homological Algebra of a
Foliation\label{SecDGHAF}}

Let $M$ be a smooth manifold and $C$ an involutive $n$-dimensional
distribution on it. Now on, I will denote by $A$ the algebra of smooth
functions on $M$. I will denote by $C\mathfrak{X}$ the submodule of
$\mathfrak{X}(M)$ made of vector fields in $C$. Let $C\Lambda^{1}%
:=C\mathfrak{X}^{\bot}\subset\Lambda^{1}(M)$ be its annihilator, and put
\[
\overline{\mathfrak{X}}:=\mathfrak{X}(M)/C\mathfrak{X},\quad\overline{\Lambda
}{}^{1}:=\Lambda^{1}(M)/C\Lambda^{1}.
\]
Then $C\Lambda^{1}\simeq\overline{\mathfrak{X}}{}^{\ast}$ and $\overline
{\Lambda}{}^{1}\simeq C\mathfrak{X}^{\ast}$. In view of the Fr\"{o}benius
theorem, there always exist coordinates $\ldots,x^{i},\ldots,u^{\alpha}%
,\ldots$, $i=1,\ldots,n$, $\alpha=1,\ldots,\dim M-n$, adapted to $C$, i.e.,
such that $C\mathfrak{X}$ is locally spanned by $\ldots,\partial_{i}%
:=\partial/\partial x^{i},\ldots$ and $C\Lambda^{1}$ is locally spanned by
$\ldots,du^{\alpha},\ldots$. Consider the Chevalley-Eilenberg algebra
$(\overline{\Lambda},\overline{d})$ of the Lie algebroid $C$. Namely,
$\overline{\Lambda}$ is the exterior algebra of $\overline{\Lambda}{}^{1}$
and
\[
(\overline{d}\lambda)(X_{1},\dots,X_{k+1})=\sum_{i}(-)^{i+1}X_{i}%
(\lambda(\ldots,\widehat{X}_{i},\ldots))+\sum_{i<j}(-)^{i+j}\lambda
([X_{i},X_{j}],\ldots,\widehat{X}_{i},\ldots,\widehat{X}_{j},\ldots),
\]
where $\lambda\in\overline{\Lambda}{}^{k}$ is understood as a $C^{\infty}%
(M)$-valued, $k$-multilinear, skew-symmetric map on $C\mathfrak{X}$ and
$X_{1},\ldots,X_{k+1}\in C\mathfrak{X}$. The DG algebra $(\overline{\Lambda
},\overline{d})$ is the quotient of $(\Lambda(M),d)$ over the differentially
closed ideal generated by $C\Lambda^{1}$ which is made of differential forms
vanishing when acting on vector fields in $C\mathfrak{X}$. In particular, it
is generated by degree $0$, and $\overline{d}$-exact degree $1$ elements. In
the following, I write $\omega\longmapsto\overline{\omega}$ the projection
$\Lambda(M)\longrightarrow\overline{\Lambda}$.

The Lie algebroid $C\mathfrak{X}$ acts on $\overline{\mathfrak{X}}$ via the
\emph{Bott connection}. Namely, write $X\longmapsto\overline{X}$ the
projection $\mathfrak{X}(M)\longrightarrow\overline{\mathfrak{X}}$. Then
\[
X{}\cdot{}\overline{Y}:=\overline{[X,Y]}\in\overline{\mathfrak{X}},\quad X\in
C\mathfrak{X},\text{ }Y\in\mathfrak{X}(M).
\]
Accordingly, there is a DG module $(\overline{\Lambda}\otimes_M\overline
{\mathfrak{X}},\overline{d})$ over $(\overline{\Lambda},\overline{d})$ whose
differential is given by the usual Chevalley-Eilenberg formula:
\begin{align*}
&  (\overline{d}Z)(X_{1},\dots,X_{k+1})\\
&  =\sum_{i}(-)^{i+1}X_{i}\cdot Z(\ldots,\widehat{X}_{i},\ldots)+\sum
_{i<j}(-)^{i+j}Z([X_{i},X_{j}],\ldots,\widehat{X}_{i},\ldots,\widehat{X}%
_{j},\ldots),
\end{align*}
where $Z\in\overline{\Lambda}{}^{k}\otimes_M\overline{\mathfrak{X}}$ is
understood as a $\overline{\mathfrak{X}}$-valued, $k$-multilinear,
skew-symmetric map on $C\mathfrak{X}$, and $X_{1},\ldots,X_{k+1}\in
C\mathfrak{X}$. The tensor product $\Lambda(M)\otimes_M\mathfrak{X}%
(M)\longrightarrow\overline{\Lambda}\otimes_M\overline{\mathfrak{X}}$ of
projections $\Lambda(M)\longrightarrow\overline{\Lambda}$ and $\mathfrak{X}%
(M)\longrightarrow\overline{\mathfrak{X}}$ will be written $Z\longmapsto
\overline{Z}$.

\begin{remark}
\label{RemExt}The differentials $\overline{d}$ in $\overline{\Lambda}$ and
$\overline{\mathfrak{X}}$ can be uniquely extended to the whole tensor algebra%
\[
\bigoplus_{i,j}\overline{\Lambda}\otimes_M\overline{\mathfrak{X}}{}^{\otimes
i}\otimes_M(C\Lambda^{1})^{\otimes j},
\]
requiring Leibniz rules with respect to tensor products and contractions. Such
extension is nothing but the Chevalley-Eilenberg differential associated to
the canonical action of $C\mathfrak{X}$ on $\bigoplus_{i,j}\overline
{\mathfrak{X}}{}^{\otimes i}\otimes_M(C\Lambda^{1})^{\otimes j}$. In particular,
$\overline{d}$ extends to an homological derivation $\overline{d}%
_{\mathcal{S}}$ of $\overline{\Lambda}\otimes_M S^{\bullet}\overline
{\mathfrak{X}}$.
\end{remark}

The exact sequence
\[
0\longrightarrow C\mathfrak{X}\longrightarrow\mathfrak{X}(M)\longrightarrow
\overline{\mathfrak{X}}\longrightarrow0
\]
splits. The datum of a splitting is equivalent to the datum of a distribution
$V$ complementary to $C$. From now on fix such a distribution. I will always
identify $\overline{\mathfrak{X}}$ (resp., $\overline{\Lambda}$) with the
corresponding submodule (resp., subalgebra) in $\mathfrak{X}(M)$ (resp.,
$\Lambda(M)$) determined by $V$.

The distribution $V\simeq TM/C$ is locally spanned by vector fields
$\ldots,V_{\alpha},\ldots$ of the form $V_{\alpha}:=\partial/\partial
u^{\alpha}+V_{\alpha}^{i}\partial_{i}$, $\alpha=1,\ldots,\dim M-n$, for some
local functions $\ldots,V_{\alpha}^{i},\ldots$. Moreover,
\begin{equation}
\lbrack\partial_{i},V_{\alpha}]=\partial_{i}V_{\alpha}^{j}\partial_{j}%
,\quad\lbrack V_{\alpha},V_{\beta}]=R_{\alpha\beta}^{i}\partial_{i},
\label{18}%
\end{equation}
where $R_{\alpha\beta}^{i}:=V_{\alpha}V_{\beta}^{i}-V_{\beta}V_{\alpha}^{i}$.

Now, consider the \emph{deformation complex} $(\mathrm{Der}\overline{\Lambda
},\delta_{0}:=[\overline{d},\cdot])$ (see, for instance, \cite{cm08}) of the
integral foliation of $C$. The complementary distribution $V$ determines
$\overline{\Lambda}$-linear contraction data $(p_{0},j_{0},h_{0})$ for
$(\mathrm{Der}\overline{\Lambda},\delta_{0})$ over $(\overline{\Lambda}%
\otimes_M\overline{\mathfrak{X}},\overline{d})$. Accordingly, there is an
$LR_{\infty}$-algebra structure on $\overline{\Lambda}\otimes_M\overline
{\mathfrak{X}}$ (see the second appendix of \cite{vi12}). Recall that the
projection $p_{0}:\mathrm{Der}\overline{\Lambda}\longrightarrow\overline
{\Lambda}\otimes_M\overline{\mathfrak{X}}$ is actually independent of $V$ and is
defined as
\begin{equation}
p_{0}\Delta:=\overline{\Delta|_{C^\infty (M)}},\quad\Delta\in\mathrm{Der}\overline
{\Lambda}.\label{19}%
\end{equation}
The injection $j_{0}:\overline{\Lambda}\otimes_M\overline{\mathfrak{X}%
}\longrightarrow\mathrm{Der}\overline{\Lambda}$ depends on $V$ and is defined
by
\[
(j_{0}Z)(\omega):=\overline{L_{Z}\omega},\quad Z\in\overline{\Lambda}%
\otimes_M\overline{\mathfrak{X}},\quad\omega\in\overline{\Lambda}.
\]
Finally, the homotopy $h_{0}:\mathrm{Der}\overline{\Lambda}\longrightarrow
\mathrm{Der}\overline{\Lambda}$ can be described as follows. First of all, I
prove a useful

\begin{lemma}
An element $\Delta\in\mathrm{Der}\overline{\Lambda}$ can be uniquely written
in the form
\begin{equation}
\Delta=i_{U}+\overline{L}_{V}+\overline{L}_{W}, \label{17}%
\end{equation}
where $U,V\in\overline{\Lambda}\otimes_M C\mathfrak{X}$, $W\in\overline{\Lambda
}\otimes_M\overline{\mathfrak{X}}$, and for $X\in\Lambda(M)\otimes_M
\mathfrak{X}(M)$ I defined $\overline{L}_{X}\in\mathrm{Der}\overline{\Lambda}$
by $\overline{L}_{X}\omega:=\overline{L_{X}\omega}$, $\omega\in\overline
{\Lambda}$.
\end{lemma}

\begin{proof}
It is easy to check the following identity
\[
\lbrack\overline{d},\overline{L}_{X}]=\overline{L}_{\overline{d}\,\overline
{X}},\quad X\in\overline{\Lambda}\otimes_M\mathfrak{X}(M).
\]
Now, let $\Delta\in\mathrm{Der}\overline{\Lambda}$, put
\[
W:=p_{0}\Delta,\quad V:=\Delta|_{C^\infty (M)}-p_{0}\Delta,\quad U:=[\Delta,\overline
{d}]|_{C^\infty (M)}+(-)^{\Delta}\overline{d}W
\]
and check (\ref{17}). It is enough to evaluate both sides of (\ref{17}) on
generators. Thus, for all $f\in C^\infty (M)$,
\[
\Delta f=(V+W)f=(i_{U}+\overline{L}_{V}+\overline{L}_{W})f
\]
Similarly,
\begin{align*}
\Delta\overline{d}f &  =[\Delta,\overline{d}]f+(-)^{\Delta}\overline{d}\Delta
f\\
&  =Uf-(-)^{\Delta}(\overline{d}W)(f)+(-)^{\Delta}\overline{d}(V+W)f\\
&  =i_{U}\overline{d}f-(-)^{\Delta}[\overline{d},\overline{L}_{W}%
]f+(-)^{\Delta}\overline{d}\overline{L}_{W}f+(-)^{\Delta}\overline{d}%
\overline{L}_{V}f\\
&  =(i_{U}+\overline{L}_{V}+\overline{L}_{W})\overline{d}f.
\end{align*}

\end{proof}

Then $h_{0}$ is given by
\[
h_{0}(i_{U}+\overline{L}_{V}+\overline{L}_{W})=(-)^{\Delta}i_{V},\quad
U,V\in\overline{\Lambda}\otimes_M C\mathfrak{X},\quad W\in\overline{\Lambda
}\otimes_M\overline{\mathfrak{X}}.
\]

\subsection{Differential Operators on a Foliated Manifold}

In $\mathcal{D}(M):=\mathcal{D}(C^\infty (M))$, consider the left ideal $\mathcal{D}%
(M)\circ C\mathfrak{X}$ generated by $C\mathfrak{X}$. Denote by $\overline
{\mathcal{D}}$ the quotient left $\mathcal{D}(M)$-module $\mathcal{D}%
(M)/\mathcal{D}(M)\circ C\mathfrak{X}$, and write $\square\longmapsto
\overline{\square}$ the projection $\mathcal{D}(M)\longrightarrow
\overline{\mathcal{D}}$. More generally, let $Q$ be the module of sections of
a vector bundle over $M$. Consider the submodule $\mathcal{D}(M,Q)\circ
C\mathfrak{X}$ in $\mathcal{D}(M,Q) := \mathcal{D}(C^\infty (M),Q)\simeq Q\otimes_M\mathcal{D}(M)$, and the
quotient $\overline{\mathcal{D}}(M,Q):=\mathcal{D}(M,Q)/\mathcal{D}(M,Q)\circ
C\mathfrak{X}$, and write again $\square\longmapsto\overline{\square}$ the
projection $\mathcal{D}(M,Q)\longrightarrow\overline{\mathcal{D}}(M,Q)$.
Clearly, $\overline{\mathcal{D}}(M,Q)\simeq Q\otimes_M\overline{\mathcal{D}}$,
and in the following I will often understand this canonical isomorphism.

The Lie algebroid $C\mathfrak{X}$ acts on $\overline{\mathcal{D}}$ as follows
\[
X\cdot\overline{\square}:=\overline{X\circ\square}=\overline{[X,\square
]},\quad X\in C\mathfrak{X},\quad\square\in\mathcal{D}(M).
\]
Notice that $\overline{\mathfrak{X}}$ can be understood as a submodule in
$\overline{\mathcal{D}}$ and the action of $C\mathfrak{X}$ on $\overline
{\mathfrak{X}}$ as the restricted action. Accordingly, the Chevalley-Eilenberg
complex $(\overline{\Lambda}\otimes_M\overline{\mathfrak{X}},\overline{d})$
extends to a Chevalley-Eilenberg complex $(\overline{\Lambda}\otimes_M
\overline{\mathcal{D}},\overline{d}_{\mathcal{D}})$ in an obvious way.

\begin{remark}
\label{Rem2}The differential
\[
\overline{d}_{\mathcal{D}}:\overline{\Lambda}\otimes_M\overline{\mathcal{D}%
}\longrightarrow\overline{\Lambda}\otimes_M\overline{\mathcal{D}}%
\]
identifies with
\[
\overline{d}_{\ast}:\overline{\mathcal{D}}(M,\overline{\Lambda})\ni
\overline{\square}\longmapsto\overline{d}_{\ast}\overline{\square}%
:=\overline{\overline{d}\circ\square}\in\overline{\mathcal{D}}(M,\overline
{\Lambda}),\quad\square\in\mathcal{D}(M,\overline{\Lambda}).
\]
Indeed, it is easy to see that both $\overline{d}_{\mathcal{D}}$ and
$\overline{d}_{\ast}$ are graded derivations subordinate to $\overline{d}$.
Therefore, it is enough to prove that they coincide on generators, namely, on
$\overline{\mathcal{D}}$. Let $\square\in\mathcal{D}(M,\overline{\Lambda
})=\overline{\Lambda}\otimes_M\mathcal{D}$. Since the isomorphism $\overline
{\Lambda}\otimes_M\mathcal{D}\longrightarrow\mathcal{D}(M,\overline{\Lambda})$
is given by $\omega\otimes\square\longmapsto\omega\square$, then
\[
\overline{\langle\square,X\rangle}=\langle\overline{\square},X\rangle,\quad
X\in C\mathfrak{X},
\]
where I indicated with $\langle W,X\rangle$ the contraction of $W\in
\overline{\Lambda}{}^{1}\otimes_M Q$ with a vector fields $X\in C\mathfrak{X}$.
Thus,%
\[
\langle\overline{d}_{\mathcal{D}}\overline{\square}|X\rangle=X\cdot
\overline{\square} =\overline{X\circ\square}=\overline{\langle\overline
{d}\circ\square|X\rangle}=\langle\overline{d}_{\ast}\overline{\square
}|X\rangle.
\]

\end{remark}

The module $\overline{\Lambda}\otimes_M\overline{\mathcal{D}}$ inherits a
filtration
\[
\overline{\Lambda}\otimes_M\overline{\mathcal{D}}_{0}\subset\overline{\Lambda
}\otimes_M\overline{\mathcal{D}}_{1}\subset\cdots\subset\overline{\Lambda
}\otimes_M\overline{\mathcal{D}}_{i}\subset\cdots\subset\overline{\Lambda
}\otimes_M\overline{\mathcal{D}}%
\]
from $\overline{\Lambda}\otimes_M\mathcal{D}$, and 1) the projection
$\mathcal{D}(M,Q)\longrightarrow\overline{\mathcal{D}}(M,Q)$, and 2) the
differential $\overline{d}_{\mathcal{D}}$, preserve this filtration.
Accordingly, the graded object $\mathrm{Gr}(\overline{\Lambda}\otimes_M
\overline{\mathcal{D}})$ identifies with $\overline{\Lambda}\otimes_M
S^{\bullet}\overline{\mathfrak{X}}$, and inherits a differential $\overline
{d}_{\mathcal{S}}$ which coincides with the one in Remark \ref{RemExt}. In
particular, $(\overline{\Lambda}\otimes_M S^{\bullet}\overline{\mathfrak{X}%
},\overline{d}_{\mathcal{S}})$ is a DG commutative algebra.

Consider again the complementary distribution $V$, and notice that, in view of
commutation relations (\ref{18}), $\overline{\mathcal{D}}$ is locally spanned
by
\[
V_{\alpha_{1}\cdots\alpha_{i}}:=\overline{V_{(\alpha_{1}}\cdots V_{\alpha
_{i})}},\quad i\geq0,
\]
and they are independent generators.

Now, consider the universal enveloping DG algebra $(\mathcal{D}(\overline
{\Lambda}),\delta_{\mathcal{D}})$ of the deformation complex $(\mathrm{Der}%
\overline{\Lambda},\delta_{0})$. In Section \ref{SecA} I show that the
contraction data $(p_{0},j_{0},h_{0})$ for $(\mathrm{Der}\overline{\Lambda
},\delta_{0})$ over $(\overline{\Lambda}\otimes_M\overline{\mathfrak{X}%
},\overline{d})$ extend to contraction data $(p,j,h)$ for $(\mathcal{D}%
(\overline{\Lambda}),\delta_{\mathcal{D}})$ over $(\overline{\Lambda}%
\otimes_M\overline{\mathcal{D}},\overline{d}_{\mathcal{D}})$. Here, I take only
two steps in this direction. Firstly, I define the projection $p:\mathcal{D}%
(\overline{\Lambda})\longrightarrow\overline{\Lambda}\otimes_M\overline
{\mathcal{D}}$, which is given by
\[
p\square:=\overline{\square|_{C^\infty (M)}},\quad\square\in\mathcal{D}(\overline
{\Lambda}),
\]
and clearly extends $p_{0}$ in (\ref{19}). Moreover, in view of Remark
\ref{Rem2} and the fact that $\overline{d}:C^\infty (M)\longrightarrow \overline
{\Lambda}$ does actually belong to $\mathcal{D}(M,\overline{\Lambda})\circ
C\mathfrak{X}$, $p:\mathcal{D}(\overline{\Lambda})\longrightarrow
\overline{\Lambda}\otimes_M\overline{\mathcal{D}}$ is a cochain map. Notice
that, as $p_{0}$, $p$ is canonical, i.e., it doesn't depend on any other
structure than the distribution $C$. Secondly, I consider the graded DG object
$(\mathcal{S}(\overline{\Lambda})\simeq S_{\overline{\Lambda}}^{\bullet
}\mathrm{Der}\overline{\Lambda},\delta_{\mathcal{S}})$ of $(\mathcal{D}%
(\overline{\Lambda}),\delta_{\mathcal{D}})$ and extend the contraction data
for $(\mathrm{Der}\overline{\Lambda},\delta_{0})$ over $(\overline{\Lambda
}\otimes_M\overline{\mathfrak{X}},\overline{d})$ to contraction data for
$(\mathcal{S}(\overline{\Lambda}),\delta_{\mathcal{S}})$ over $(S_{\overline
{\Lambda}}^{\bullet}(\overline{\Lambda}\otimes_M\overline{\mathfrak{X}}%
)\simeq\overline{\Lambda}\otimes_M S^{\bullet}\overline{\mathfrak{X}}%
,\overline{d}_{\mathcal{S}})$ as in Section \ref{SecHTUE}. The next step is to
construct \textquotedblleft PBW isomorphisms\textquotedblright\
\[
\overline{\Lambda}\otimes_M\overline{\mathcal{D}}\approx\overline{\Lambda
}\otimes_M S^{\bullet}\overline{\mathfrak{X}},\quad\mathcal{D}(\overline
{\Lambda})\approx\mathcal{S}(\overline{\Lambda}).
\]
This can be done exploiting an \emph{adapted connection}. I devote the next
section to the introduction of this geometric structure.

\subsection{Adapted Connections}

In this section $C,V$ are complementary distributions on $M$. I don't require
$C$ to be involutive. The above definitions of $C\mathfrak{X}$, $\overline
{\mathfrak{X}}$, $C\Lambda^{1}$, and $\overline{\Lambda}{}^{1}$ are still
valid in the present general situation. Moreover, let
\begin{align*}
\mathfrak{X}(M) &  \ni X\longmapsto CX\in C\mathfrak{X}\\
\Lambda^{1}(M) &  \ni\omega\longmapsto\omega^{C}\in C\Lambda^{1}%
\end{align*}
be the projections. The pair $(C,V)$ determines a distinguished class of
connections according to the following

\begin{definition}
The connection $\nabla$ is called \emph{adapted to the pair }$(C,V)$ (or
simply \emph{adapted}) if

\begin{enumerate}
\item it restricts to $\overline{\Lambda}{}^{1}$, i.e., $\nabla_{X}\omega
\in\overline{\Lambda}{}^{1}$ for all $X\in\mathfrak{X}(M)$ and $\omega
\in\overline{\Lambda}{}^{1}$,

\item it restricts to $C\Lambda^{1}$, i.e., $\nabla_{X}\omega\in C\Lambda^{1}$
for all $X\in\mathfrak{X}(M)$ and $\omega\in C\Lambda^{1}$,

\item $\nabla_{Y}\omega=\overline{L_{Y}\omega}$ for all $Y\in\overline
{\mathfrak{X}}$, $\omega\in\overline{\Lambda}{}^{1}$,

\item $\nabla_{X}\omega=(L_{X}\omega)^{C}$ for all $X\in C\mathfrak{X}$,
$\omega\in C\Lambda^{1}$.
\end{enumerate}
\end{definition}

\begin{proposition}
There exist adapted connections.
\end{proposition}

\begin{proof}
Let $\tilde{\nabla}$ be a fiducial connection. For $X\in\mathfrak{X}(M)$ and
$\omega\in\Lambda^{1}(M)$ put
\begin{equation}
\nabla_{X}\omega=((\tilde{\nabla}_{\overline{X}}+L_{CX})\omega^{C}%
)^{C}+\overline{(\tilde{\nabla}_{CX}+L_{\overline{X}})\overline{\omega}%
}.\label{9}%
\end{equation}
The operator $\nabla_{X}$ is clearly a derivation subordinate to $X$, i.e.,
$\nabla_{X}f\omega=X(f)\omega+f\nabla_{X}\omega$. Moreover, $\nabla_{X}$ is
$C^\infty (M)$-linear in $X$. Indeed, for $f\in C^\infty (M)$
\begin{align*}
L_{CfX}\omega^{C} &  =fL_{CX}\omega^{C}+df\wedge i_{CX}\omega^{C}%
=fL_{CX}\omega^{C}\\
L_{f\overline{X}}\overline{\omega} &  =fL_{\overline{X}}\overline{\omega
}+df\wedge i_{\overline{X}}\overline{\omega}=fL_{\overline{X}}\overline
{\omega}.
\end{align*}
Thus, the correspondence $X\longmapsto\nabla_{X}$ is a linear connection. The
four properties of adapted connections are obvious.
\end{proof}

\begin{proposition}
Let $\nabla$ be an adapted connection determined by a connection
$\tilde{\nabla}$ via Formula (\ref{9}). Then
\begin{equation}
\nabla_{X}Y=\overline{(\tilde{\nabla}_{\overline{X}}+L_{CX})\overline{Y}%
}+C(\tilde{\nabla}_{CX}+L_{\overline{X}})CY. \label{8}%
\end{equation}
In particular,

\begin{enumerate}
\item $\nabla$ restricts to $\overline{\mathfrak{X}}$, i.e., $\nabla_{X}%
Y\in\overline{\mathfrak{X}}$ for all $X\in\mathfrak{X}(M)$ and $Y\in
\overline{\mathfrak{X}}$,

\item $\nabla$ restricts to $C\mathfrak{X}$, i.e., $\nabla_{X}Y\in
C\mathfrak{X}$ for all $X\in\mathfrak{X}(M)$ and $Y\in C\mathfrak{X}$,

\item $\nabla_{Y}X=C[Y,X]$, and $\nabla_{X}Y=\overline{[Y,X]}$ for all
$Y\in\overline{\mathfrak{X}}$, $X\in C\mathfrak{X}$.
\end{enumerate}
\end{proposition}

\begin{proof}
Let $X,Y\in\mathfrak{X}(M)$, and $\omega\in\Lambda^{1}(M)$.
\begin{align*}
\langle\nabla_{X}Y,\omega\rangle &  =X\langle Y,\omega\rangle-\langle
Y,\nabla_{X}\omega\rangle\\
&  =X\langle Y,\omega\rangle-\langle\overline{Y},(\tilde{\nabla}_{\overline
{X}}+L_{CX})\omega^{C}\rangle-\langle CY,(\tilde{\nabla}_{CX}+L_{\overline{X}%
})\overline{\omega}\rangle\\
&  =X\langle Y,\omega\rangle-X\langle\overline{Y},\omega^{C}\rangle
+\langle(\tilde{\nabla}_{\overline{X}}+L_{CX})\overline{Y},\omega^{C}%
\rangle-X\langle CY,\overline{\omega}\rangle+\langle(\tilde{\nabla}%
_{CX}+L_{\overline{X}})CY,\overline{\omega}\rangle\\
&  =\langle\overline{(\tilde{\nabla}_{\overline{X}}+L_{CX})\overline{Y}%
}+C(\tilde{\nabla}_{CX}+L_{\overline{X}})CY,\omega\rangle.
\end{align*}

\end{proof}

It is easy to see that every adapted connection is of the form (\ref{9}): for
an adapted connection $\nabla$ it is enough to put $\tilde{\nabla}=\nabla$.
Indeed,
\begin{align*}
((\nabla_{\overline{X}}+L_{CX})\omega^{C})^{C}+\overline{(\nabla
_{CX}+L_{\overline{X}})\overline{\omega}}  &  =(\nabla_{\overline{X}}%
+L_{CX})\omega^{C}+(\nabla_{CX}+L_{\overline{X}})\overline{\omega}\\
&  =\nabla_{X}\omega^{C}+\nabla_{X}\overline{\omega}\\
&  =\nabla_{X}\omega.
\end{align*}

\begin{proposition}
Let $\nabla$ be an adapted connection determined by a connection
$\tilde{\nabla}$ with torsion $\tilde{T}\in\Lambda^{2}(M)\otimes_M
\mathfrak{X}(M)$. The torsion $T$ of $\nabla$ is given by
\[
T(X,Y)=\overline{\tilde{T}(\overline{X},\overline{Y})}+C(\tilde{T}%
(CX,CY))-\overline{[CX,CY]}-C[\overline{X},\overline{Y}],\quad X,Y\in
\mathfrak{X.}%
\]

\end{proposition}

\begin{proof}
Compute
\begin{align*}
T(X,Y)  &  =\nabla_{X}Y-\nabla_{Y}X-[X,Y]\\
&  =\overline{(\tilde{\nabla}_{\overline{X}}+L_{CX})\overline{Y}}%
+C(\tilde{\nabla}_{CX}+L_{\overline{X}})CY\\
&  \quad\ -\overline{(\tilde{\nabla}_{\overline{Y}}+L_{CY})\overline{X}%
}-C(\tilde{\nabla}_{CY}+L_{\overline{Y}})CX-[X,Y]\\
&  =\overline{\tilde{T}(\overline{X},\overline{Y})}+C\tilde{T}%
(CX,CY)+\overline{[\overline{X},\overline{Y}]}+C[CX,CY]\\
&  \quad\ +\overline{[CX,\overline{Y}]}+C[\overline{X},CY]+\overline
{[\overline{X},CY]}+C[CX,\overline{Y}]-[X,Y]\\
&  =\overline{\tilde{T}(\overline{X},\overline{Y})}+C\tilde{T}%
(CX,CY)-\overline{[CX,CY]}-C[\overline{X},\overline{Y}].
\end{align*}

\end{proof}

\begin{corollary}
A torsion-free adapted connection exists iff both $C$ and $V$ are involutive.
\end{corollary}

\begin{proof}
If both $C$ and $V$ are involutive, an adapted connection determined by a
torsion-free connection is torsion-free as well. Conversely, let the adapted
connection $\nabla$ determined by a connection $\tilde{\nabla}$ be
torsion-free. Then, for $X,Y\in\overline{\mathfrak{X}}$,
\[
0=T(X,Y)=\overline{\tilde{T}(X,Y)}-C[X,Y]\Longrightarrow C[X,Y]=0.
\]
Similarly, for $X,Y\in C\mathfrak{X}$.
\end{proof}

\begin{definition}
An adapted connection $\nabla$ is called\emph{ torsion-quasi-free} if
\[
T(X,Y)=-\overline{[CX,CY]}-C[\overline{X},\overline{Y}].
\]

\end{definition}

\begin{corollary}
There exist torsion-quasi-free adapted connections.
\end{corollary}

\begin{proof}
The adapted connection determined by a torsion-free connection is torsion-quasi-free.
\end{proof}

Now, suppose that $C$ is involutive and let $\nabla$ be a torsion-quasi-free
adapted connection. Let $T$ be the torsion of $\nabla$. Then, clearly,

\begin{enumerate}
\item $\nabla$ extends the Bott connection,

\item $T$ coincides with the curvature form of $V$, up to a sign,

\item In view of (\ref{18})
\begin{equation}
T(V_{\alpha},V_{\beta})=-R_{\alpha\beta}^{i}\partial_{i}. \label{20}%
\end{equation}

\end{enumerate}

\subsection{Two PBW Isomorphisms\label{PBW}}

Now, let $C$ be again an involutive distribution on $M$, and $\nabla$ a
connection in $\Lambda^{1}(M)$ adapted to the pair $(C,V)$ and
torsion-quasi-free. The connection $\nabla$ determines two PBW type
isomorphisms (see \cite{lsx12} for a similar result)
\[
\underline{\mathsf{PBW}}:\overline{\Lambda}\otimes_M\overline{\mathcal{D}%
}\approx\overline{\Lambda}\otimes_M S^{\bullet}\overline{\mathfrak{X}}%
,\quad\mathsf{PBW}:\mathcal{D}(\overline{\Lambda})\approx\mathcal{S}%
(\overline{\Lambda})
\]
as follows. For $\omega\in\overline{\Lambda}$ and $P\in S^{\bullet}%
\overline{\mathfrak{X}}$, put
\[
\underline{\mathsf{PBW}}(\omega\otimes P):=\omega\otimes\overline{\nabla_{P}%
},
\]
where $\nabla_{P}$ is defined as in Example \ref{Examp}. To define
$\mathsf{PBW}$ notice, first of all, that every derivation $\Delta
\in\mathrm{Der}\overline{\Lambda}$ can be uniquely written in the form
\begin{equation}
\Delta=i_{W}+\nabla_{V}+\overline{L}_{Z}, \label{21}%
\end{equation}
$W,V\in\overline{\Lambda}\otimes_M C\mathfrak{X}$, $Z\in\overline{\Lambda
}\otimes_M\overline{\mathfrak{X}}$, where $\nabla_{V}$ is defined as in Example
\ref{Examp}. Indeed, let
\[
\Delta=i_{U}+\overline{L}_{V}+\overline{L}_{Z},
\]
$U,V\in\overline{\Lambda}\otimes_M C\mathfrak{X}$, $Z\in\overline{\Lambda
}\otimes_M\overline{\mathfrak{X}}$. Now, since $\nabla$ is adapted, and
torsion-quasi-free, then $\nabla_V$ preserves $\overline{\Lambda}$, and $\overline{L}_{V}=\nabla_{V}+i_{\nabla V}$, with
$\nabla V\in\overline{\Lambda}\otimes_M C\mathfrak{X}$. Thus (\ref{21}) holds
simply putting $W:=U+\nabla V$. Clearly, the correspondence
\[
\left(\overline{\Lambda}\otimes_M C\mathfrak{X}[1]\right) \oplus \left(\overline{\Lambda}\otimes_M
C\mathfrak{X}\right) \oplus \left(\overline{\Lambda}\otimes_M\overline{\mathfrak{X}}\right)%
\ni(W,V,Z)\longmapsto i_{W}+\nabla_{V}+\overline{L}_{Z}\in\mathrm{Der}%
\overline{\Lambda}%
\]
is an isomorphism of $\overline{\Lambda}$-module, so that
\[
\mathcal{S}(\overline{\Lambda})\simeq\overline{\Lambda}\otimes_M\Lambda
^{\bullet}C\mathfrak{X}\otimes_M S^{\bullet}C\mathfrak{X}\otimes_M S^{\bullet
}\overline{\mathfrak{X}}%
\]
and $p_{0}:\mathcal{S}(\overline{\Lambda})\longrightarrow\overline{\Lambda
}\otimes_M S^{\bullet}\overline{\mathfrak{X}}$ is the obvious projection.
Moreover, let
\[
\Sigma=\omega\otimes X_{1}\wedge\cdots\wedge X_{k}\otimes P\otimes
Q\in\mathcal{S}(\overline{\Lambda}),
\]
put
\[
\mathsf{PBW}(\Sigma):=\omega i_{X_{1}}\cdots i_{X_{k}}\nabla_{P\odot Q}%
\in\mathcal{D}(\overline{\Lambda}).
\]

\begin{remark}
\label{Rempj0}In general, the isomorphisms $\underline{\mathsf{PBW}}$ and
$\mathsf{PBW}$ are not compatible with projections $p,p_{0}$. However, if one
uses them to induce an injection $j_{0}^{\sim}:\overline{\Lambda}%
\otimes_M\overline{\mathcal{D}}\longrightarrow\mathcal{D}(\overline{\Lambda})$,
from the injection $j_{0}:\overline{\Lambda}\otimes_M S^{\bullet}\overline
{\mathfrak{X}}\longrightarrow\mathcal{S}(\overline{\Lambda})$, then the former
is a right inverse of $p$. Indeed, clearly
\[
\underline{\mathsf{PBW}}=p\circ\mathsf{PBW}\circ j_{0},
\]
so that, for $\Sigma\in\overline{\Lambda}\otimes_M S^{i}\overline{\mathfrak{X}%
}$
\begin{align*}
(p\circ j_{0}^{\sim}\circ\underline{\mathsf{PBW}})\Sigma &  =(p\circ
j_{0}^{\sim}\circ p\circ\mathsf{PBW}\circ j_{0})\Sigma\\
&  =(p\circ\mathsf{PBW}\circ j_{0}\circ\sigma_{i}\circ p\circ\mathsf{PBW}\circ
j_{0})\Sigma\\
&  =(p\circ\mathsf{PBW}\circ j_{0}\circ p_{0}\circ\sigma_{i}\circ
\mathsf{PBW}\circ j_{0})\Sigma\\
&  =(p\circ\mathsf{PBW}\circ j_{0}\circ p_{0}\circ j_{0})\Sigma\\
&  =(p\circ\mathsf{PBW}\circ j_{0})\Sigma\\
&  =\underline{\mathsf{PBW}}(\Sigma).
\end{align*}

\end{remark}

In the following, I will often understand isomorphisms $\underline
{\mathsf{PBW}}$ and $\mathsf{PBW}$.

\section{The $A_{\infty}$-Algebra of a Foliation\label{SecA}}

Let $M$ be a smooth manifold and let $C$ be an involutive distribution on it.
Summarizing results obtained so far, a complementary distribution $V$ and a
torsion-quasi-free adapted connection $\nabla$ determine

\begin{enumerate}
\item Contraction data $(p_{0},j_{0},h_{0})$ for $(\mathcal{S}(\overline
{\Lambda}),\delta_{\mathcal{S}})$ over $(\overline{\Lambda}\otimes_M S^{\bullet
}\overline{\mathfrak{X}},\overline{d}_{\mathcal{S}})$,

\item PBW type isomorphisms $\overline{\Lambda}\otimes_M\overline{\mathcal{D}%
}\approx\overline{\Lambda}\otimes_M S^{\bullet}\overline{\mathfrak{X}}$,
$\mathcal{D}(\overline{\Lambda})\approx\mathcal{S}(\overline{\Lambda})$,
\end{enumerate}

Notice that, actually, i) $p_{0}$ is independent of the supplementary
geometric data $V$ and $\nabla$, and 2) $j_{0},h_{0}$ do only depend on $V$.

Now, put $t=\delta_{\mathcal{S}}-\delta_{\mathcal{D}}:\mathcal{D}%
(\overline{\Lambda})\longrightarrow\mathcal{D}(\overline{\Lambda})$. The
Perturbation Lemma determines a \textquotedblleft new\textquotedblright%
\ differential $\overline{d}_{t}:\overline{\Lambda}\otimes_M\overline
{\mathcal{D}}\longrightarrow\overline{\Lambda}\otimes_M\overline{\mathcal{D}}$
and contraction data $(p_{t},j_{t},h_{t})$ for $(\mathcal{D}(\overline
{\Lambda}),\delta_{\mathcal{D}})$ over $(\overline{\Lambda}\otimes_M
\overline{\mathcal{D}},\overline{d}_{t})$, given by Formulas (\ref{p}%
), (\ref{j}), (\ref{hi}). In its turn, the Homotopy Transfer Theorem determine
an $A_{\infty}$-algebra structure on $(\overline{\Lambda}\otimes_M
\overline{\mathcal{D}},\overline{d}_{t})$. Before giving more details about
these structures, I remark that $p_{t}$ is actually independent of $V$ and
$\nabla$ and coincides with the canonical projection $p:\mathcal{D}%
(\overline{\Lambda})\longrightarrow\overline{\Lambda}\otimes_M\overline
{\mathcal{D}}$. To show this, first notice that $pj_{0}=\mathrm{id}$ and
$ph_{0}=0$ (the first identity is discussed in Remark \ref{Rempj0}, while the
second one is immediate from the definitions of $h_{0}$ and $p$). It follows
that $pj_{t}=\mathrm{id}$ and $ph_{t}=0$. Now, let $\square\in\mathcal{D}%
(\overline{\Lambda})$. Then
\[
0=p[h_{t},\delta]\square=p(\mathrm{id}-j_{t}p_{t})\square=(p-p_{t})\square.
\]
As an immediate consequence, $\overline{d}_{t}$ is also independent of $V$ and
$\nabla$, and coincides with the canonical differential $\overline
{d}_{\mathcal{D}}:\overline{\Lambda}\otimes_M\overline{\mathcal{D}%
}\longrightarrow\overline{\Lambda}\otimes_M\overline{\mathcal{D}}$. In the
following, I put $j:=j_{t}$ and $h:=h_{t}$.

I am finally in the position to furnish few details about the (strict unital)
$A_{\infty}$-algebra structure $\{\alpha_{k},\ k\in\mathbb{N}\}$ on
$\overline{\Lambda}\otimes_M\overline{\mathcal{D}}$. To this aim, notice that
the isomorphism $\mathsf{PBW}:\mathcal{D}(\overline{\Lambda})\approx
\mathcal{S}_{\bullet}(\overline{\Lambda})$ (resp., $\underline{\mathsf{PBW}%
}:\overline{\Lambda}\otimes_M\overline{\mathcal{D}}\approx\overline{\Lambda
}\otimes_M S^{\bullet}\overline{\mathfrak{X}}$) determines a new grading in
$\mathcal{D}(\overline{\Lambda})$ (resp., $\overline{\Lambda}\otimes_M
\overline{\mathcal{D}}$), which I call the \emph{order }and is given by the
decomposition $\mathcal{S}(\overline{\Lambda})=\bigoplus_{k}\mathcal{S}%
_{k}(\overline{\Lambda})$ (resp., $\overline{\Lambda}\otimes_M S^{\bullet
}\overline{\mathfrak{X}}=\bigoplus_{k}\overline{\Lambda}\otimes_M S^{k}%
\overline{\mathfrak{X}}$). Every map $\phi$ of the spaces $\mathcal{D}%
(\overline{\Lambda})$ and $\overline{\Lambda}\otimes_M\overline{\mathcal{D}}$
have its homogenous components with respect to the order. I denote by
$\phi^{\lbrack i]}$ the $i$-th one, and by $\mathcal{O}(i)$ a generic (no
better specified) object of order no higher than $i$, e.g.,
\[
\delta_{\mathcal{D}}=\delta_{\mathcal{S}}+\mathcal{O}(-1){},\quad\overline
{d}_{\mathcal{D}}=\overline{d}_{\mathcal{S}}{}+\mathcal{O}(-1),\quad
p=p_{0}+\mathcal{O}(-1),\quad h={}h_{0}+\mathcal{O}(-1).
\]
Similarly,
\[
t=t^{[-1]}+\mathcal{O}(-2),
\]
and
\begin{equation}
j=j_{0}+h_{0}t^{[-1]}j_{0}+\mathcal{O}(-2).\label{j-1}%
\end{equation}
I will not need to compute $t^{[-1]}$. Finally, the composition $\circ$ of
differential operators in $\mathcal{D}(\overline{\Lambda})$, decomposes as
\[
{}\circ{}={}\odot{}+{}\circledast{}+\mathcal{O}(-2){},
\]
where I put $\circledast{}:={}\circ^{\lbrack-1]}$.

Notice that, in view of the above decompositions of $\delta_{\mathcal{D}%
},\overline{d}_{\mathcal{D}}$ and the contraction data $(p,j,h)$, the $k$-th
Poisson bracket in $\overline{\Lambda}\otimes_M S^{\bullet}\overline
{\mathfrak{X}}$ is the skew-symmetrization $\mathsf{A}\alpha_{k}^{[1-k]}$ of
$\alpha_{k}^{[1-k]}$. In particular, the skew-symmetrization of $\alpha
_{k}^{[1-k]}$ vanishes for $k>3$ \cite{vi12}. My next aim is twofold:

\begin{enumerate}
\item proving that $\alpha_{k}$ has no component of order higher than
$\alpha_{k}^{[1-k]}$, i.e., $\alpha_{k}=\alpha_{k}^{[1-k]}+\mathcal{O}(-k)$,
for $k\neq2$,

\item \textquotedblleft computing\textquotedblright\ $\alpha_{k}^{[1-k]}$ and,
in particular, showing that it is zero for $k>3$.
\end{enumerate}

Notice that the first claim states that the order of $\alpha_{k}(\square
_{1},\ldots,\square_{k})$ is no higher than $1-k+\sum_{i}\ell_{i}$ for
$\square_{i}=$ $\mathcal{O}(\ell_{i})$, $i=1,\ldots,k$. The claim that
$\alpha_{k}^{[1-k]}=0$ for $k>3$, can be interpreted as a further motivation
why the $LR_{\infty}$-algebra structure on $\overline{\Lambda}\otimes_M
\overline{\mathfrak{X}}$ presents just one higher homotopy \cite{vi12}. In
order to reach my aim, I first prove a

\begin{lemma}
The order $-1$ component of the projection $p$ vanishes, i.e., $p^{[-1]}=0$
(so that $p=p_{0}+\mathcal{O}(-2)$).
\end{lemma}

\begin{proof}
Let $\square\in\mathcal{D}(\overline{\Lambda})$ be of order $H$. Then,
$\square$ is locally of the form
\[
\square=\sum_{k+\ell+m=H}A^{i_{1}\cdots i_{k}|j_{1}\cdots j_{\ell}|\alpha
_{1}\cdots\alpha_{m}}I_{i_{1}\cdots i_{k}}\nabla_{j_{1}}\cdots\nabla_{j_{\ell
}}\nabla_{\alpha_{1}}\cdots\nabla_{\alpha_{m}}%
\]
where $I_{i_{1}\cdots i_{k}}:=i_{\partial_{i_{1}}}\cdots i_{\partial_{i_{k}}}%
$, and the $A$'s are components of a (contravariant) tensor. The $A$'s are
skew-symmetric in the $i$'s, symmetric in the $j$'s and symmetric in the
$\alpha$'s. Compute
\[
p\square=\sum_{\ell+m=H}A^{\varnothing|j_{1}\cdots j_{\ell}|\alpha_{1}%
\cdots\alpha_{m}}\overline{\nabla_{j_{1}}\cdots\nabla_{j_{\ell}}\nabla
_{\alpha_{1}}\cdots\nabla_{\alpha_{m}}}%
\]
Clearly,
\[
p^{[0]}\square=(p\square)^{[H]}=p_{0}\square=A^{\varnothing|\varnothing
|\alpha_{1}\cdots\alpha_{H}}V_{\alpha_{1}\cdots\alpha_{H}}.
\]
Now, let $\ell>0$,
\begin{align*}
&  A^{\varnothing|j_{1}\cdots j_{\ell}|\alpha_{1}\cdots\alpha_{m}}%
\overline{\nabla_{j_{1}}\cdots\nabla_{j_{\ell}}\nabla_{\alpha_{1}}\cdots
\nabla_{\alpha_{m}}}\\
&  =A^{\varnothing|j_{1}\cdots j_{\ell}|\alpha_{1}\cdots\alpha_{m}}%
\overline{\nabla_{j_{1}}\cdots\nabla_{j_{\ell-1}}[\nabla_{j_{\ell}}%
,\nabla_{\alpha_{1}}\cdots\nabla_{\alpha_{m}}]}\\
&  =A^{\varnothing|j_{1}\cdots j_{\ell}|\alpha_{1}\cdots\alpha_{m}}\sum_{r\leq
m}\overline{\nabla_{j_{1}}\cdots\nabla_{j_{\ell-1}}\nabla_{\alpha_{1}}%
\cdots\underset{r}{[\nabla_{j_{\ell}},\nabla_{\alpha_{r}}]}\cdots
\nabla_{\alpha_{m}}}.
\end{align*}
Since $\nabla$ is adapted and torsion-quasi-free
\begin{equation}
\lbrack\nabla_{i},\nabla_{\alpha}]\lambda_{\beta_{1}\cdots\beta_{t}}%
=\sum_{s\leq t}R_{i\alpha\beta_{s}}^{\nabla}{}^{\beta}\lambda_{\beta_{1}%
\cdots\underset{s}{\beta}\cdots\beta_{t}},\label{22}%
\end{equation}
for all covariant tensors $\lambda$ locally of the form
\[
\lambda=\lambda_{\beta_{1}\cdots\beta_{t}}du^{\beta_{1}}\otimes\cdots\otimes
du^{\beta_{t}}.
\]
In (\ref{22}) $R^{\nabla}$ is the curvature tensor of $\nabla$. It follows
from (\ref{22}), that $[\nabla_{i},\nabla_{\alpha}]=\mathcal{O}(0)$, and
\[
(\overline{\nabla_{(j_{1}}\cdots\nabla_{j_{\ell})}\nabla_{(\alpha_{1}}%
\cdots\nabla_{\alpha_{m})}})\in\mathcal{O}(\ell+m-2)
\]
for all $\ell,m$. I conclude that
\[
p\square=p_{0}\square+\mathcal{O}(H-2).
\]

\end{proof}

The following proposition is a corollary of the above lemma, and the side
conditions $ph=0$, $hj=0$, $h^{2}=0$.

\begin{proposition}%
\begin{align*}
\gamma_{k} &  =\mathcal{O}(1-k),\quad k\geq1\\
\beta_{k} &  =\mathcal{O}(2-k),\quad k\geq2\\
\alpha_{k} &  =\mathcal{O}(1-k),\quad k\geq3
\end{align*}
while
\[
\alpha_{2}=\mathcal{O}(0).
\]
Moreover, the highest order componet\ $\alpha_{k}^{[1-k]}$ of $\alpha_{k}$ can
be computed iteratively via formulas
\begin{align*}
\varepsilon_{k}(\square_{1},\ldots,\square_{k}) &  :=-\sum_{\ell
+m=k}(-)^{a(\ell,m,\boldsymbol{\square})}\gamma_{\ell}^{[1-\ell]}(\square
_{1},\ldots,\square_{\ell})\circledast\gamma_{m}^{[1-m]}(\square_{\ell
+1},\ldots,\square_{\ell+m})\\
\gamma_{k}^{[1-k]} &  =h_{0}\varepsilon_{k},\\
\alpha_{k}^{[1-k]} &  =p_{0}\varepsilon_{k},
\end{align*}
$\boldsymbol{\square}=(\square_{1},\ldots,\square_{k})$, $\square_1,\ldots,\square_k \in \overline{\Lambda
}\otimes_M\overline{\mathcal{D}}$, being a $k$-tuple of homogeneous
elements of given orders, $k\geq2$.
\end{proposition}

\begin{proof}
The two parts of the proposition can be checked simultaneously by induction on
$k$. Indeed, $\gamma_{1}=-j=-j_{0}+\mathcal{O}(-1)$, $\beta_{2}=j(-)\circ
j(-)=j_{0}(-\odot-)+\mathcal{O}(-1)$, and $\alpha_{2}=(-\odot-{}%
)+\mathcal{O}(-1)$ (where I used that $p_{0}$ preserves the product $\odot$).
Moreover,
\[
\gamma_{2}=h\beta_{2}=h(j(-)\circ j(-)),
\]
so that
\[
\gamma_{2}^{[0]}=h_{0}j_{0}(-\odot-)=0.
\]
Thus, compute
\begin{align*}
\gamma_{2}^{[-1]} &  =h^{[-1]}j_{0}(-\odot-)+h_{0}(j^{[-1]}(-)\odot
j_{0}(-)+j_{0}(-)\circledast j_{0}(-)+j_{0}(-)\odot j^{[-1]}(-))\\
&  =h_{0}(j_{0}(-)\circledast j_{0}(-))
\end{align*}
where I used formulas (\ref{hi}), (\ref{j-1}). Now,
\[
\beta_{k}=\sum_{\ell+m=k}(-)^{\ell-1}\gamma_{\ell}(-)\circ\gamma_{m}(-)
\]
with $\gamma_{\ell}=\mathcal{O}(1-\ell)$ and $\gamma_{m}=\mathcal{O}(1-m)$ by
induction hypothesis. Therefore, it is immediately seen that $\beta
_{k}=\mathcal{O}(2-k)$, and
\[
\beta_{k}^{[2-k]}=\sum_{\ell+m=k}(-)^{\ell-1}\gamma_{\ell}^{[1-\ell]}%
(-)\odot\gamma_{m}^{[1-m]}(-),
\]
so that
\[
\gamma_{k}=h\beta_{k}=\mathcal{O}(2-k).
\]
But
\begin{align*}
\gamma_{k}^{[2-k]} &  =\sum_{\ell+m=k}(-)^{\ell-1}h_{0}(\gamma_{\ell}%
^{[1-\ell]}(-)\odot\gamma_{m}^{[1-m]}(-))\\
&  =\sum_{\ell+m=k}(-)^{\ell-1}h_{0}(h_{0}\varepsilon_{\ell}(-)\odot
h_{0}\varepsilon_{m}(-))\\
&  =0.
\end{align*}
Now, compute
\[
\beta_{k}^{[1-k]}=\sum_{\ell+m=k}(-)^{\ell-1}(\gamma_{\ell}^{[-\ell]}%
(-)\odot\gamma_{m}^{[1-m]}(-)+\gamma_{\ell}^{[1-\ell]}(-)\circledast\gamma
_{m}^{[1-m]}(-)+\gamma_{\ell}^{[1-\ell]}(-)\odot\gamma_{m}^{[-m]}(-)),
\]
and%
\[
\gamma_{k}^{[1-k]}=h^{[-1]}\beta_{k}^{[2-k]}+h_{0}\beta_{k}^{[1-k]}=h_{0}%
\beta_{k}^{[1-k]}=h_{0}\varepsilon_{k},
\]
where I used (\ref{hi}).

Finally, compute
\[
\alpha_{k}=p\beta_{k}=\mathcal{O}(2-k).
\]
But
\[
\alpha_{k}^{[2-k]}=p_{0}\beta_{k}^{[2-k]}=\sum_{\ell+m=k}(-)^{\ell-1}%
p_{0}\gamma_{\ell}^{[1-\ell]}(-)\odot p_{0}\gamma_{m}^{[1-m]}(-)=0,
\]
where I used the side condition $p_{0}h_{0}=0$, and
\[
\alpha_{k}^{[1-k]}=p^{[-1]}\beta_{k}^{[2-k]}+p_{0}\beta_{k}^{[1-k]}=p_{0}%
\beta_{k}^{[1-k]}=p_{0}\varepsilon_{k},
\]
where I used the above lemma and the side condition $p_{0}h_{0}=0$ again.
\end{proof}

In view of the above proposition, a formula for $\circledast$ is enough to get
inductive formulas for the $\alpha_{k}^{[1-k]}$'s. These formulas, which I
compute in the proof of the next lemma, actually show that $\alpha_{k}%
^{[1-k]}=0$ for $k>3$, as announced.

Now on put
\[
\mathcal{S}_{i,j,\ell}:=\overline{\Lambda}\otimes_M\Lambda^{i}C\mathfrak{X}%
\otimes_M S^{j}C\mathfrak{X}\otimes_M S^{\ell}\overline{\mathfrak{X}}%
\subset\mathcal{S}(\overline{\Lambda})
\]

\begin{lemma}
Let $\square_{1}\in\mathcal{S}_{r,0,\ell}$ and $\square_{2}\in\mathcal{S}%
_{s,0,m}$, then
\begin{align*}
\square_{1}\circledast\square_{2}  &  \in\mathcal{S}_{r+s,1,\ell
+m-2}+\mathcal{S}_{r+s,0,\ell+m-1}+\mathcal{S}_{r+s-1,0,\ell+m}\\
h_{0}(\square_{1}\circledast\square_{2})  &  \in\mathcal{S}_{r+s+1,0,m+\ell
-2}\\
p_{0}(\square_{1}\circledast\square_{2})  &  \in\left\{
\begin{array}
[c]{cc}%
\overline{\Lambda}\otimes_M S^{\ell+m-1}\overline{\mathfrak{X}} & \text{if
}r+s=0\\
\overline{\Lambda}\otimes_M S^{\ell+m}\overline{\mathfrak{X}} & \text{if
}r+s=1\\
0 & \text{if }r+s>1
\end{array}
\right.
\end{align*}

\end{lemma}

\begin{proof}
The operators $\square_{1}$ and $\square_{2}$ are locally of the form
\begin{align*}
\square_{1}  &  =\Phi^{i_{1}\cdots i_{r}|\alpha_{1}\cdots\alpha_{\ell}%
}I_{i_{1}\cdots i_{r}}\nabla_{\alpha_{1}}\cdots\nabla_{\alpha_{\ell}}, \\
\square_{2}  &  =\Psi^{j_{1}\cdots j_{s}|\beta_{1}\cdots\beta_{m}}%
I_{j_{1}\cdots j_{s}}\nabla_{\beta_{1}}\cdots\nabla_{\beta_{m}}.
\end{align*}
Then
\begin{align*}
&  \square_{1}\circ\square_{2}\\
&  =\Phi^{i_{1}\cdots i_{r}|\alpha_{1}\cdots\alpha_{\ell}}\Psi^{j_{1}\cdots
j_{s}|\beta_{1}\cdots\beta_{m}}I_{i_{1}\cdots i_{r}j_{1}\cdots j_{s}}%
\nabla_{(\alpha_{1}}\cdots\nabla_{\alpha_{\ell}}\nabla_{\beta_{1}}\cdots
\nabla_{\beta_{m})}\\
&  \quad\ +\Phi^{i_{1}\cdots i_{r}|\alpha_{1}\cdots\alpha_{\ell}}\Psi
^{j_{1}\cdots j_{s}|\beta_{1}\cdots\beta_{m}}I_{i_{1}\cdots i_{r}j_{1}\cdots
j_{s}}(\nabla_{(\alpha_{1}}\cdots\nabla_{\alpha_{\ell})}\nabla_{(\beta_{1}%
}\cdots\nabla_{\beta_{m})})^{[\ell+m-1]}\\
&  \quad\ +\ell\Phi^{i_{1}\cdots i_{r}|\alpha_{1}\cdots\alpha_{\ell}}%
\nabla_{\alpha_{1}}\Psi^{j_{1}\cdots j_{s}|\beta_{1}\cdots\beta_{m}}%
I_{i_{1}\cdots i_{r}j_{1}\cdots j_{s}}\nabla_{(\alpha_{2}}\cdots\nabla
_{\alpha_{\ell}}\nabla_{\beta_{1}}\cdots\nabla_{\beta_{m})}\\
&  \quad\ +r\Phi^{i_{1}\cdots i_{r}|\alpha_{1}\cdots\alpha_{\ell}}I_{i_{1}%
}\Psi^{j_{1}\cdots j_{s}|\beta_{1}\cdots\beta_{m}}I_{i_{2}\cdots i_{r}%
j_{1}\cdots j_{s}}\nabla_{(\alpha_{1}}\cdots\nabla_{\alpha_{\ell}}%
\nabla_{\beta_{1}}\cdots\nabla_{\beta_{m})}+O(\ell+m-2).
\end{align*}
It remains to compute
\[
(\nabla_{(\alpha_{1}}\cdots\nabla_{\alpha_{\ell})}\nabla_{(\beta_{1}}%
\cdots\nabla_{\beta_{m})})^{[\ell+m-1]}.
\]
Let $A^{\alpha_{1}\cdots\alpha_{\ell}|\beta_{1}\cdots\beta_{m}}$ be symmetric
in the $\alpha$'s and the $\beta$'s separately. Then
\begin{align*}
&  A^{\alpha_{1}\cdots\alpha_{\ell}|\beta_{1}\cdots\beta_{m}}(\nabla
_{(\alpha_{1}}\cdots\nabla_{\alpha_{\ell})}\nabla_{(\beta_{1}}\cdots
\nabla_{\beta_{m})})^{[\ell+m-1]}\\
&  =A^{\alpha_{1}\cdots\alpha_{\ell}|\beta_{1}\cdots\beta_{m}}(\nabla
_{\alpha_{1}}\cdots\nabla_{\alpha_{\ell}}\nabla_{\beta_{1}}\cdots\nabla
_{\beta_{m}})^{[\ell+m-1]}\\
&  =\dfrac{m}{m+1}A^{\alpha\alpha_{1}\cdots\alpha_{\ell-1}|\beta\beta
_{1}\cdots\beta_{m-1}}R_{\alpha\beta}^{i}\nabla_{i}\nabla_{(\alpha_{1}}%
\cdots\nabla_{\alpha_{\ell-1}}\nabla_{\beta_{1}}\cdots\nabla_{\beta_{m-1})}.
\end{align*}

I conclude that
\begin{align}
\square_{1}\circledast\square_{2}  &  =\dfrac{m}{m+1}R_{\alpha\beta}^{i}%
\Phi^{i_{1}\cdots i_{r}|\alpha\alpha_{1}\cdots\alpha_{\ell-1}}\Psi
^{j_{1}\cdots j_{s}|\beta\alpha_{\ell}\cdots\beta_{\ell+m-2}}i_{i_{1}\cdots
i_{r}j_{1}\cdots j_{s}}\nabla_{i}\nabla_{(\alpha_{1}}\cdots\nabla
_{\alpha_{\ell+m-2})}\nonumber\\
&  \quad\ +\ell\Phi^{i_{1}\cdots i_{r}|\alpha\alpha_{1}\cdots\alpha_{\ell-1}%
}\nabla_{\alpha}\Psi^{j_{1}\cdots j_{s}|\alpha_{\ell}\cdots\alpha_{\ell+m-1}%
}i_{i_{1}\cdots i_{r}j_{1}\cdots j_{s}}\nabla_{(\alpha_{1}}\cdots
\nabla_{\alpha_{\ell+m-1})}\nonumber\\
&  \quad\ +r\Phi^{ii_{1}\cdots i_{r-1}|\alpha_{1}\cdots\alpha_{\ell}}i_{i}%
\Psi^{j_{1}\cdots j_{s}|\alpha_{\ell+1}\cdots\alpha_{\ell+m}}i_{i_{1}\cdots
i_{r-1}j_{1}\cdots j_{s}}\nabla_{(\alpha_{1}}\cdots\nabla_{\alpha_{\ell+m})}
\label{B}%
\end{align}

\end{proof}

\begin{corollary}
Let $\square_{1},\ldots,\square_{k}\in\overline{\Lambda}\otimes_M\overline
{\mathcal{D}}$ with $\square_{i}$ being of order $\ell_{i}$, $i=1,\ldots,k$.
Put $\ell:=\ell_{1}+\cdots+\ell_{k}$. Then
\begin{align*}
\gamma_{k}^{[1-k]}(\square_{1},\ldots,\square_{k}) &  \in\mathcal{S}%
_{k-1,0,\ell-2k+2},\\
\varepsilon_{k}(\square_{1},\ldots,\square_{k}) &  \in\mathcal{S}%
_{k-2,1,\ell-2k+2}+\mathcal{S}_{k-2,0,\ell-2k+3}+\mathcal{S}_{k-3,0,\ell
-2k+4},\\
\alpha_{k}^{[1-k]}(\square_{1},\ldots,\square_{k}) &  \in\left\{
\begin{array}
[c]{cc}%
\overline{\Lambda}\otimes_M S^{\ell-1}\overline{\mathfrak{X}} & \text{if }k=2\\
\overline{\Lambda}\otimes_M S^{\ell-2}\overline{\mathfrak{X}} & \text{if }k=3\\
0 & \text{if }k>3
\end{array}
\right.  ,
\end{align*}
$k>1$.
\end{corollary}

\begin{proof}
Immediate, by induction on $k.$
\end{proof}

Now, compute $\alpha_{k}^{[1-k]}$, for $k=1,2,3$. Let $\square_{1},\square
_{2},\square_{3}\in\overline{\Lambda}\otimes_M\overline{\mathcal{D}}$ be locally
given by
\[
\square_{i}=\Phi_{i}^{\alpha_{1}\cdots\alpha_{r}}V_{\alpha_{1}\cdots\alpha
_{r}},\quad i=1,2,3.
\]
First of all,
\[
\alpha_{2}^{[0]}(\square_{1},\square_{2})=\Phi_{1}^{\alpha_{1}\cdots\alpha
_{r}}\Phi_{2}^{\alpha_{r+1}\cdots\alpha_{r+s}}V_{\alpha_{1}\cdots\alpha_{r+s}%
}.
\]
Moreover, using Formula (\ref{B}), it is easy to see that
\begin{align*}
\alpha_{2}^{[-1]}(\square_{1},\square_{2})  &  =r\Phi_{1}^{\alpha\alpha
_{1}\cdots\alpha_{r-1}}\nabla_{\alpha}\Phi_{2}^{\alpha_{r}\cdots\alpha
_{r+s-1}}V_{\alpha_{1}\cdots\alpha_{r+s-1}}\\
\alpha_{3}^{[-2]}(\square_{1},\square_{2},\square_{3})  &  =\dfrac{2t}%
{t+1}R_{\alpha\beta}^{i}\Phi_{1}^{\alpha\alpha_{1}\cdots\alpha_{r-1}}\Phi
_{2}^{\beta\alpha_{r}\cdots\alpha_{r+s-2}}I_{i}\Phi_{3}^{\alpha_{r+s-1}%
\cdots\alpha_{r+s+t-2}}V_{\alpha_{1}\cdots\alpha_{r+s+t-2}}%
\end{align*}
which are duly consistent with formulas in \cite{vi12}.

\begin{remark}
Notice that the natural $\mathcal{D}(\overline{\Lambda})$-module structure on
$\overline{\Lambda}$ can be transferred along the contraction data $(p,j,h)$
as well. Indeed, $\overline{\Lambda}$ is actually a DG module over
$\mathcal{D}(\overline{\Lambda})$ with differential $\overline{d}%
:\overline{\Lambda}\longrightarrow\overline{\Lambda}$. Moreover, this DG
module structure (and the DG algebra structure on $\mathcal{D}(\overline
{\Lambda})$) can be encoded in a DG algebra structure on $\mathcal{D}%
(\overline{\Lambda})\oplus\overline{\Lambda}$ given by
\[
(\square_{1},\omega_{1})(\square_{2},\omega_{2}):=(\square_{1}\circ\square
_{2},\square_{1}\omega_{2}),\quad(\square_{i},\omega_{i})\in\mathcal{D}%
(\overline{\Lambda})\oplus\overline{\Lambda},\quad i=1,2,
\]
with differential $\delta^{\oplus}:=\delta_{\mathcal{D}}\oplus\overline{d}$.
Similarly, consider the complex $((\overline{\Lambda}\otimes_M\overline
{\mathcal{D}})\oplus\overline{\Lambda},\overline{d}{}^{\oplus})$ where
$\overline{d}{}^{\oplus}:=\overline{d}{}_{\mathcal{D}}\oplus\overline{d}$.
There are obvious contraction data $(p^{\oplus},j^{\oplus},h^{\oplus})$ of
$(\mathcal{D}(\overline{\Lambda})\oplus\overline{\Lambda},$ $\delta^{\oplus})$
over $((\overline{\Lambda}\otimes_M\overline{\mathcal{D}})\oplus\overline{\Lambda
},\overline{d}{}^{\oplus})$. Namely,
\[
p^{\oplus}:=p\oplus\mathrm{id},\quad j^{\oplus}:=j\oplus\mathrm{id},\quad
h^{\oplus}:=h\oplus0.
\]
Accordingly, there is an $A_{\infty}$-algebra structure $\{\alpha_{k}^{\oplus
},\ k\in\mathbb{N}\}$ in $\overline{\Lambda}\otimes_M\overline{\mathcal{D}%
}\oplus\overline{\Lambda}$. By construction,
\[
\alpha_{k}^{\oplus}((\square_{1},\omega_{1}),\ldots,(\square_{k},\omega
_{k}))=\alpha_{k}(\square_{1},\ldots,\square_{k})+\alpha_{k}^{\oplus}%
(\square_{1},\ldots,\square_{k-1},\omega_{k}),
\]
$(\square_{i},\omega_{i})\in\mathcal{D}(\overline{\Lambda})\oplus
\overline{\Lambda}$, $i=1,\ldots,k$. Therefore, if one puts
\[
\mu_{k}(\square_{1},\ldots,\square_{k-1}|\omega):=\alpha_{k}^{\oplus}%
(\square_{1},\ldots,\square_{k-1},\omega),
\]
then $\{\mu_{k},\ k\in\mathbb{N}\}$ is an $A_{\infty}$-module structure on
$\overline{\Lambda}$, over the $A_{\infty}$-algebra $\overline{\Lambda}%
\otimes_M\overline{\mathcal{D}}$. It is easy to see that, since $h^{\oplus}%
\rho=0$ for all $\rho\in\overline{\Lambda}$, then the $\mu$'s are simply given
by
\begin{align*}
\mu_{1} &  =\overline{d}\\
\mu_{k}(\square_{1},\ldots,\square_{k-1}|\omega) &  =(-)^{k-1}\gamma
_{k}(\square_{1},\ldots,\square_{k-1})\omega,\quad k\geq2.
\end{align*}

\end{remark}

\section*{Conclusions}

I proved that the $LR_{\infty}$-algebra $(\overline{\Lambda}\otimes_M
\overline{\mathfrak{X}},\mathscr{L})$ of a foliation \cite{vi12} can be
actually extended in a natural way to an $A_{\infty}$-algebra $(\overline
{\Lambda}\otimes_M\overline{\mathcal{D}},\mathscr{A})$ of longitudinal
form-valued normal differential operators. This can be done via purely
geometric data, namely a distribution complementary to the characteristic
distribution and a connection (of a suitable kind). Notice that $(\overline
{\Lambda}\otimes_M\overline{\mathfrak{X}},\mathscr{L})$ can be interpreted (to
some extent) as the (derived) Lie-Rinehart algebra of vector fields on the
space $\boldsymbol{P}$ of integral manifolds. Similarly, it is natural to
interpret $(\overline{\Lambda}\otimes_M\overline{\mathcal{D}},\mathscr{A})$ as
the (derived) associative algebra of differential operators on $\boldsymbol{P}$.
In this respect, it is tempting to conjecture that $(\overline{\Lambda}%
\otimes_M\overline{\mathcal{D}},\mathscr{A})$ is a universal enveloping SH
algebra of $(\overline{\Lambda}\otimes_M\overline{\mathfrak{X}},\mathscr{L})$.
However, the theory of universal enveloping of $LR_{\infty}$-algebras (or
$L_{\infty}$-algebroids) is not yet available and developing this research
line goes beyond the scopes of this paper. Here, I just notice that
$(\overline{\Lambda}\otimes_M\overline{\mathcal{D}},\mathscr{A})$ is indeed a
(possibly non universal) \emph{enveloping SH algebra} of $(\overline{\Lambda
}\otimes_M\overline{\mathfrak{X}},\mathscr{L})$ in the following sense. The
inclusion $\overline{\Lambda}\otimes_M\overline{\mathfrak{X}}\longrightarrow
\overline{\Lambda}\otimes_M\overline{\mathcal{D}}$ can be trivially extended to
a morphism $J:\overline{\Lambda}\otimes_M\overline{\mathfrak{X}}\longrightarrow
\overline{\Lambda}\otimes_M\overline{\mathcal{D}}$ of the $L_{\infty}$-algebra
$(\overline{\Lambda}\otimes_M\overline{\mathfrak{X}},\mathscr{L})$ and the
$L_{\infty}$-algebra obtained by skew-symmetrization of operations in
$\mathscr{A}$, simply putting $J_{k}=0$ for $k>1$. Then, it is easy to see,
using the explicit fomulas for brackets in $\mathscr{L}$ \cite{vi12}, that
\[
\nu_{k}(Z_{1},\ldots,Z_{k-1}|\omega)=\ (\mathsf{A}\alpha_{k})(Z_{1}%
,\ldots,Z_{k-1},\omega),\quad\omega\in\overline{\Lambda},\quad Z_{i}%
\in\overline{\Lambda}\otimes_M\overline{\mathfrak{X}},
\]
$i=1,\ldots,k-1$, which specializes (\ref{23}) to the present simple case
where $j$ is an inclusion and $J_{k}=0$ for $k>1$.

\end{document}